\documentclass[
11pt]{amsart}
\usepackage[margin=4 cm]{geometry}
\makeatletter
\@namedef{subjclassname@2020}{%
  \textup{2020} Mathematics Subject Classification}
\makeatother

\usepackage{lineno}
\usepackage[centertags]{amsmath}
\usepackage{amsfonts}
\usepackage{amssymb}
\usepackage{amsthm}
\usepackage{xcolor}

\newtheorem{thm}{Theorem}[section]
\newtheorem{prop}[thm]{Proposition}
\newtheorem{definition}[thm]{Definition}
\newtheorem{lemma}[thm]{Lemma}
\newtheorem{remark}[thm]{Remark}
\newtheorem{cor}[thm]{Corollary}
\numberwithin{equation}{section} 

\def\BB{\mathcal{B}}
\def\CC{\mathcal{C}}

\def\HH{\mathcal{H}}

\def\KK{\mathcal{K}}
\def\LL{\mathcal{L}}

\def\OO{\mathcal{O}}

\def\SS{\mathcal{S}}
\def \TT{\mathcal {T}}
\def\UU{\mathcal{U}}

\def\WW{\mathcal{W}}

\newcommand\bC{{\mathbb C}}

\newcommand\bF{{\mathbb F}}

\newcommand\bR{{\mathbb R}}

\newcommand\bZ{{\mathbb Z}}

\def\real{\mathbb{R}}
\def\integer{\mathbb{Z}}
\def\complex{\mathbb{C}}
\def\id{\mathrm{id}}
\newcommand{\musrb}{\mu_{\tiny{\mbox{SRB}}}}
\def\supp{\mathrm{supp}\, }
\def\D{\mathrm {d}}

\def\E{e}
\newcommand\Erg{{\operatorname{Erg}}}
\newcommand{\bi}{{\bf{i}}}
\newcommand{\bj}{{\bf{j}}}
\newcommand{\tT}{\widetilde T}
\newcommand{\tO}{\widetilde O}
\newcommand{\htop}{h_{\scriptsize{\mbox{top}}}}
\newcommand{\Fix}{\operatorname{Fix}}
\newcommand\optop{\operatorname{top}}

\newcommand\mult{\operatorname{mult}}
\newcommand\opleb{\operatorname{Leb}}
\newcommand\diam{\operatorname{diam}}

\begin{document}
\title[Thermodynamic formalism for piecewise expanding maps]{Thermodynamic formalism  for 
piecewise expanding maps in finite dimension}
\author{Viviane Baladi\textsuperscript{(1),(2)} \and Roberto Castorrini\textsuperscript{(3)}}
\address{(1) Sorbonne Universit\'e and Universit\'e Paris Cit\'e, CNRS,   Laboratoire de Probabilit\'es, Statistique et Mod\'elisation,
	 F-75005 Paris, France}
\address{(2) Institute for Theoretical Studies, ETH, 8092 Z\"urich,
	Switzerland} 

\email{baladi@lpsm.paris}
\address{(3) Dipartimento di Matematica, Universit\`a di Pisa, Largo B. Pontecorvo 5, 56127 Pisa,  Italy}
\email{roberto.castorrini@dm.unipi.it}
\date{March 2, 2024, to appear DCDS}
\thanks{It is a pleasure to thank Y.~Guedes-Bonthonneau, J.~Buzzi, M.~Demers, D.-J.~Feng, M.~Misiurewicz,
T.~Persson,  D.~Smania,
 I.P.~Toth, and, especially, M.~J\'ez\'equel,  for helpful comments. Part of this work was done  while VB was visiting Lund University
on a Knut and Alice Wallenberg fellowship in 2021,    ITS-ETHZ Zurich in 2022 and 2023, and the Mittag Leffler Institute in 2023. This work was started when RC
was working at LPSM, CNRS, Paris.  We are grateful to the referees for a careful reading.
This research is partially supported
by the European Research Council (ERC) under the European Union's Horizon 2020 research and innovation programme (grant agreement No 787304).  RC is partially supported by
the research project PRIN 2017S35EHN of the Italian Ministry of Education and Research. RC acknowledges membership to the GNFM/INDAM.} 

\begin{abstract}
For $\bar \alpha >1$ and $\alpha \in (0, \bar \alpha]$,
we study  weighted transfer operators associated to a piecewise expanding $\CC^{\bar \alpha}$ map $T$ on a compact manifold
of dimension $d\ge 1$, and a piecewise $\CC^\alpha$ weight $g$,
acting on  Sobolev spaces.
We  bound the essential spectral radius  in terms of a topological
pressure for a subadditive potential.
 Under  a new small
boundary pressure condition, we improve the estimate
by establishing a  variational principle for  piecewise expanding maps and subadditive potentials. 
\end{abstract}


\maketitle

\section{Introduction}

\subsection{Functional Approach to Ergodic Properties}
For $M$  a connected compact Riemannian manifold and $T:M\to M$,  the {\it functional analytic} approach to statistical properties of the dynamics $T$
consists in
finding a Banach space $\BB$ of functions
or distributions on $M$ such that the  {\it (Ruelle) transfer operator }
\[
\LL_{T,g} \varphi(x)= \sum_{Ty=x} g(y)\varphi(y)\, ,\quad
x \in M\, , 
\] 
weighted by a suitable function $g:M\to \bC$, and defined
initially on a subset of measurable functions $\varphi :M\to \complex$,
extends to a bounded operator on  $\BB$ on which its
\emph{essential spectral radius}\footnote{The radius of the smallest
disc outside of which the spectrum of $\LL_{T,g}$
consists of isolated eigenvalues of finite multiplicity.
Although the essential spectral radius 
depends on $\BB$, these eigenvalues
in general do not, see e.g. \cite[Thm~2.3]{Go2}.} 
 is smaller than its spectral radius (``quasicompactness'').

If $g$ is positive and $T$ is mixing, the  spectral picture can sometimes
be strengthened as follows:
The transfer operator has a positive maximal eigenvalue,
which is the exponential 
$\E^{P(\log g)}$ of the topological pressure
of $\log g$. This eigenvalue is simple, and the rest of the spectrum is
contained in a smaller disc. This
``spectral gap'' often implies  existence, uniqueness and  decay of correlations (for suitable observables) of the equilibrium
state of $\log g$, i.e. the invariant measure 
maximising $h_\mu+\int \log g \D\mu$ (where $h_\mu$ is
the Kolmogorov entropy). For $g=|\det DT|^{-1}$, we have
in many cases that $\E^{P_{\optop}(\log g)}=1$, and
the equilibrium state of $\log g$ is the physical (SRB) measure.

\smallskip

Another desirable goal (besides  finding a Banach space
on which the essential spectral radius is small) is  to relate  the isolated eigenvalues of the transfer operator 
with the poles of a 
dynamical zeta function defined  by
assuming that $\Fix T^n=\{x \in M\mid T^n(x)=x\}$
is finite for each fixed $n$,  and setting (in the sense of formal power series)
$
\zeta_{T, g}(z)=\exp \sum_{n=1}^{\infty} \frac{z^{n}}{n} 
\sum_{x \in \Fix T^n} g^{(n)}(x)$.
We hope that the
Milnor--Thurston  kneading operator approach of \cite{BaRu} (see \cite{BaTsu1} or \cite[\S 3.2]{Babook2} for an implementation
to smooth dynamics in arbitrary dimension)
can be applied to piecewise expanding or piecewise hyperbolic dynamics in
arbitrary dimension.

\subsection{(Piecewise) Expanding Case}

For
\emph{expanding and piecewise  expanding} maps $T$
(with smallest expansion denoted by $\lambda>1$),  the relevant $\BB$ is a space of functions. In the {\it smooth} expanding case, the pioneering bounds
of Ruelle \cite{Ru2} on the essential spectral radius,
taking $\BB$ the space of H\"older functions, were shown to be optimal by Gundlach--Latushkin
\cite{GuLa}, who reformulated them using a variational (thermodynamic)
expression.  
The 
{\it piecewise} expanding theory is  fairly complete in one-dimension, usually 
taking $\BB$ the set $BV$ of functions of bounded variation (piecewise monotonicity
is enough there,
see e.g. \cite{Babook1, BaRu}).

For \emph{higher dimensional}
piecewise expanding dynamics,
quasicompactness  (and even ergodic properties such as existence of the SRB measure)
can fail  \cite{Tsu, Buz3} if one does not 
make \emph{additional assumptions on the  ``complexity at the beginning''} $D^b(T)$ (also called
``entropy multiplicity,'' 
see \eqref{complexity-cond}): The works \cite{BaGo9,BaGo10,Tho,Liv13}
require  some version of  $D^b (T) < \log \lambda$
(``hyperbolicity beats  complexity at the beginning'')  to  bound the essential spectral radius. Cowieson \cite{Cow} proved that  $D^b(T)=0$  for $T$
in an open and dense subset of piecewise $\CC^{\bar\alpha}$ expanding maps,
and that $D^b(T)=0$ implies a spectral gap for the operator associated to $g=|\det DT|^{-1}$
acting on  $\BB=BV$ (see \cite{Saus} for a different choice of $\BB$).
For arbitrary piecewise $\CC^\alpha$
weights $g$, in any dimension, Thomine \cite{Tho}, inspired by \cite{BaGo9}, obtained\footnote{Using a tower construction,
	Buzzi et al. \cite{BPS, BM} had previously
	found mild additional conditions implying exponential decay of correlations
	for H\"older observables.} a bound (see \eqref{DDT}) on the essential spectral
radius on classical Sobolev spaces
$\BB=\HH^t_p$, for $1<p<\infty$ and $0<t<\min\{\alpha, 1/p\}$.
Even if $D^b(T)=0$ and $g= |\det DT|^{-1}$, Thomine's bound ensures
quasicompactness  only if either 
$T$ satisfies some pinching condition (for example if $T$ is a $\beta$ transformation)
or $p$ is close
enough to $1$, in order to control
the exponential growth of the number of preimages of $T^n$ (the ``complexity
at the end'', see Remark~\ref{Dend}).
Liverani \cite[(3)]{Liv13} found sufficient conditions  ensuring a spectral gap\footnote{The assumption
\cite[(3)]{Liv13} that ``hyperbolicity beats  complexity'' is similar to   requiring $D^b_{\{\nu_n\}}<0$, see \eqref{complexity-condw}.}
on $\BB=BV$ 
 for  $d\ge 1$ and  piecewise expanding maps  having possibly  infinitely many domains of continuity
and (controlled) blowup of derivatives, with $g=|\det DT|^{-1}$.
 The only results  linking the
poles of the zeta function to the spectrum
of $\LL_{T,g}$ for higher dimensional\footnote{In the more general setting of
non-degenerate  entropy expanding maps, Buzzi \cite[\S6]{Buz10} found a domain of
meromorphy for $\zeta_{T,g}(z)$ if $g\equiv 1$ (but no spectral interpretation of the poles).
}
piecewise expanding maps were obtained by Buzzi and Keller \cite{BuKe}, for   piecewise affine  maps and $g=|\det DT|^{-1}$.

\subsection{(Piecewise) Hyperbolic Case}

 The pioneering work \cite{BlKeLi} 
introduced \emph{aniso\-tro\-pic} Banach spaces $\BB$ of {\it distributions,} adapted to 
Anosov diffeomorphisms $T$.
These spaces can be classified
in two categories (see \cite{Ba17}): \emph{Geometric spaces,}
using integration over stable submanifolds,
and \emph{micro-local spaces,}
using the Fourier transform.
The nature of geometric norms (taking
a supremum over a class of submanifolds) does not seem to be
amenable to the Milnor--Thurston approach.
In the 
\emph{smooth} hyperbolic case,
the best known estimate  for the essential spectral radius is 
obtained for micro-local spaces by \emph{thermodynamic formalism} techniques,
as a variational expression
for a subadditive topological pressure
\cite{BaTsu1}.

Many physically relevant models, such as dispersive billiards are  uniformly hyperbolic, but only piecewise smooth. 
The geometric approach \cite{BlKeLi, GoLi} has been
used to study the SRB measure of  piecewise hyperbolic maps with controlled complexity in  dimension two (\cite{DeLi}),
but also the SRB measure and other equilibrium states of Sinai billiard maps and flows  (\cite{DeZa, BaDeLi,BaDe1, BaDe2, BCD}). 
It has recently been extended to the random
Lorentz gas,
via  Birkhoff cones \cite{DeLi2}.
Estimates on the essential spectral
radius for micro-local spaces  were\footnote{Even  if $D^b(T)=0$, 
\cite{BaGo9, BaGo10}
do not always give
 quasicompactness if $g\ne  |\det DT|^{-1}$.} obtained (\cite{BaGo9, BaGo10}) for weighted piecewise hyperbolic surface
maps. (The results there do not apply
 to Sinai billiards, for which the derivative is unbounded.)
We are not aware of any result
linking the poles of dynamical zeta functions with
the spectrum of transfer operators for piecewise
hyperbolic maps.

A modification $\UU^{t,s}_p$ of the micro-local spaces of \cite{BaGo10}
 suitable in the piecewise smooth setting has been proposed in
\cite{Ba17, Ba18}.  J\'ez\'equel, observed that, even if
 $D^b(T)=0$ and $g= |\det DT|^{-1}$, 
the bound on the essential
spectral radius of \cite[Thm~4.1]{Ba17} may not imply quasicompactness:
 For
 a linear automorphism $T$ of the two-torus with expanding eigenvalue $\Lambda >1$,
 the essential spectral radius of the operator for $|\det DT|^{-1} \equiv 1$
 acting on the space $\UU^{t,s}_p$ from \cite{Ba17} is bounded by
 $r_0(t,s,p)= \frac{\Lambda}{\Lambda^{1/p}}
 \max \{\Lambda^{-t}, \Lambda^{t+s}\}$
 (the  factor $\Lambda$ comes from a naive use of ``complexity at the end'').
 To ensure that
 characteristic functions  are bounded multipliers \cite[Thm~3.1]{Ba18}, we must take $-1+ 1/p <s<-t<0$,
so that
 $r_0(t,s,p)>1$. Our hope is that a ``thermodynamic'' control of the complexity at the end (using fragmentation
and reconstitution, as below) will replace $r_0(t,s,p)$ by
$r(t,s,p)= \frac{\Lambda^{\tilde p} }{\Lambda^{1/p}}
 \max \{\Lambda^{-t}, \Lambda^{t+s}\}$ for some $\tilde p\in (0,1)$
 (probably $(p-1)/p$, in view of \eqref{DDT} and 
 \cite[Thm~2.5]{BaGo10}), allowing  parameters for which
$r(t,s,p)<1$.

\subsection{Outline of the Results}

We consider the toy-model case of
weighted  piecewise expanding maps  and classical (isotropic) Sobolev spaces $\HH^t_p$, just like in \cite{Tho}, but
thermodynamic estimates replace the 
``complexity at the end'':
Our first main result, Theorem~\ref{thm:ess-bound}, 
 gives an unconditional bound   \eqref{eq:ess-bound1} 
on the essential
spectral radius in terms of the topological pressure 
of a subadditive potential related to $\log |g|$.
(We recover the optimal bound \cite{BaKe, Kel} in dimension one. We improve
on Thomine's bound \cite{Tho} in the generic small boundary entropy case $D^b(T)=0$. 
If $g=|\det DT|^{-1}$, our bound is analogous to Liverani's bound \cite[Thm~1, Lemma~3.1]{Liv13}
for the essential spectral radius.
Our bound coincides with   Gundlach--Latushkin's 
\cite{GuLa}
bound on H\"older spaces
if the map and weight are smooth. See Remarks~\ref{d1}, \ref{compLiv}, \ref{Dend2}, \ref{GL},  \ref{condest}.)
Assuming small boundary pressure, a variational principle of Buzzi--Sarig \cite{Buz}
allows us to reformulate  \eqref{eq:ess-bound1} in Corollary~\ref{cor1}.

Next, Theorem~\ref{thm:var-princ} generalises
 this additive variational principle   \cite{Buz}
to a class of subadditive potentials
 (subadditive potentials appear naturally in dynamics, for example
$\log |\det DT|$ in dimension two or higher, see below --- our results are  the piecewise smooth analogue of
 \cite[\S3]{BaTsu1}, see also \cite[App.~B]{Babook2}). 
Combining Theorem~\ref{thm:var-princ} with Theorem~\ref{thm:ess-bound} 
yields Corollary~\ref{lecor}, which gives the variational expression \eqref{eq:ess-bound2}
for the  bound \eqref{eq:ess-bound1},  under 
   a  new subadditive small boundary pressure condition.
Our results are strongest in the SRB case
$g=|\det DT|^{-1}$, letting $1/p> t$ both tend to $1$. 

One of the features of our approach is 
fragmentation-reconstitution Lemma~\ref{lem:frag-rec}, which allows us to conveniently
use a zoom for arbitrary values of our parameter $p$. 
We hope that our results  lay the groundwork for the 
implementation of the ``ultimate'' micro-local Banach space $\UU^{t,s}_p$
from \cite{Ba17, Ba18}
in the  setting of 
piecewise hyperbolic systems,  
giving  also information on 
zeta functions.

\subsection{Outline of the Paper}
The paper is organised as follows:  In Section~\ref{sec:setting}, after defining our class of
piecewise $\CC^{\bar\alpha}$ expanding maps $T$
and piecewise $\CC^{\alpha}$ weights $g$, we state
 our two main results: Theorem~\ref{thm:ess-bound} on the essential spectral radius of the weighted transfer operator
 $\LL_g$ and Theorem~\ref{thm:var-princ} on the subadditive variational principle. 
 We  state and prove Corollary~\ref{cor1} and Corollary~\ref{lecor}, which follow
 from Theorem~\ref{thm:ess-bound}  and, respectively,  \eqref{laref} and Theorem~\ref{thm:var-princ},
 and give conditional  variational expressions for the bound on the essential spectral radius. In Section~\ref{sec:PartI},  we establish Theorem~\ref{thm:ess-bound}. 
 For this, we  prove the key Lasota--Yorke inequality, Proposition~\ref{lem:main-LY}, in \S\ref{sec:LY} and exploit it in \S\ref{sec:proof-thm}. 
Then, Section~\ref{sec:pressure}
contains the (independent)  proof  of Theorem~\ref{thm:var-princ}, adapting
 \cite{Buz},
using symbolic dynamics
 and a variational principle of Cao--Feng--Huang \cite{CFH} for continuous subadditive potentials.

\section{Setting, Definitions, and Precise Statement of Results}
\label{sec:setting}
\subsection{Piecewise Expanding Maps}
\label{sec:def} 
Throughout, $M$ is a compact  connected $\CC^{\infty}$ Riemannian manifold of dimension $d <\infty$, and $\TT_x$ denotes the tangent space of $M$ at $x\in M$. 
If $M$ has  a boundary we let $\widetilde M$ be
 a compact  connected $\CC^{\infty}$ Riemannian manifold of dimension $d <\infty$ containing the union of $M$ and a small neighbourhood of its boundary,
 otherwise we take $\widetilde M=M$. For noninteger $\beta >1$, we denote by $\CC^\beta$ those
$\CC^{[\beta]}$
 maps  whose partial derivatives of order $[\beta]$
are $(\beta-[\beta])$-H\"older.  
If a map $F$ is invertible on a set $E$, we write $F^{-1}|_{E}=(F|_{E})^{-1}$,
abusing notation.
Fixing real numbers
$\bar  \alpha > 1$ and $0 <\alpha \le \bar \alpha$,
we introduce our object of study:

\begin{definition}\label{object}
A map
$T:M \to M$ is called \emph{piecewise ($\CC^{\bar\alpha}$) expanding} if there exists a finite set  of pairwise disjoint open sets $\OO=\{ O_i \}_{i\in I}$, covering Lebesgue almost all $M$, such that each $\partial O_i$ is a finite union of $\CC^1$ compact hypersurfaces with boundaries, and moreover, for each $i\in  I$ there exists a neighbourhood $\tO_i$ of $\overline O_i$
in $\widetilde M$ and a $\CC^{\bar\alpha}$ 
diffeomorphism $\widetilde T_i: \tO_i\to T_i(\tO_i)\subset \widetilde M$  such that $T|_{O_i}=\widetilde T_i |_{O_i}$,
and, setting 
$\lambda_i(x)= \inf_{v\in \TT_x M\setminus{\{0\}}}\frac{\| D_x \widetilde T_i v \|}{\|v\|}$
for  $i\in I$ and    $x \in \tO_i$,
\begin{equation}
 \lambda=\inf_{i\in I}\inf_{x\in \tO_i}\lambda_i(x)>1\, .
\end{equation}
\end{definition}

\begin{remark}\label{refY}
Using a Taylor series, our assumption implies that, for any $\lambda' \in (1,\lambda)$,  there exists $\epsilon'>0$ such that, refining  $\OO$ to a finite collection $\OO'=\{ O'_i \}_{i\in I'}$
(such that each $\partial O'_i$ is a finite union of $\CC^1$ compact hypersurfaces with boundaries)
of pairwise disjoint open sets of diameter smaller than $\epsilon'$
covering Lebesgue almost all $M$, we have 
$$d(T(x), T(y)) \ge  \lambda' d(x,y), \,\quad
\forall  x, y \in O'_i\,,\,\,\forall i \, .
$$ 
From now on, we assume that such a refinement has
been done, using the notation $\lambda$, $\OO$, $O_i$, $I$, for  $\lambda'$, $\OO'$, $O'_i$,  $I'$.
\end{remark}

 For  $T$ as in Definition~\ref{object} (and Remark~\ref{refY}),  
 we introduce,
for
$n\ge 1$ and  $\bi=(i_0, \ldots , i_{n-1})\in I^n$,
the $n$-cylinder $O_{\bi}$ by
\[
{ 
O_\bi=O_{\left(i_{0}, \ldots, i_{n}\right)}
=\bigcap_{k=0}^{n-1}T^{-k}O_{i_k}\, .
}
\]
Note that for each $n\ge 1$ and almost every $x\in M$, there exists a unique $\bi\in I^n$ such that $x\in O_\bi$.
The corresponding (mod-$0$) partition into $n$-cylinders $O_{\bi}$ is  denoted by $\OO^{(n)}$ (so that $\OO=\OO^{(1)}$). We set $\diam(\OO^{(n)})=\max_{O_{\bi}\in \OO^{(n)}}\diam(O_\bi)$, so  that\footnote{Recall Remark~\ref{refY}, and proceed inductively on $n$, noting that
$\bigcap_{k=0}^{n-1}T^{-k}O_{i_k}=O_{i_0}\bigcap T^{-1}(\bigcap_{k=1}^{n-1}T^{-(k-1)}O_{i_k})$.}
 $\diam(\OO^{(n)})\le \lambda^{-n}\diam(M)$. 

 For $n\ge 1$ and $\bi \in I^n$ with $O_\bi\ne \emptyset$, the map $\widetilde T_{\bi}^n=\widetilde T_{i_{n-1}}\circ \cdots\circ \widetilde T_{i_{0}}$ is defined in the neighbourhood $ \tO_{\bi}=\bigcap_{k=0}^{n-1}\widetilde T^{-k}_{(i_{0}, \ldots, i_{k-1})}\tO_{i_k}$ of $\overline O_{\bi}$ (we put $\overline O_{\bi}=\emptyset$
 if $O_\bi=\emptyset$). Setting
\begin{equation}
\label{singset}
\partial \OO=\cup_i \partial O_i \, , \quad \SS_\OO=\cup_{k\ge 0}T^{-k}(\partial \OO)\, ,
\end{equation}
(note that $\SS_\OO$ has zero Lebesgue measure), we  put 
 \begin{align}\label{eq:hyp-index}
&\nu_{n }(x)=\frac 1{ \inf_{\|v\|=1} \|D_{x}T^{n} v\|}
 \, , \, \, x\in M\setminus \SS_\OO\, ,\\
\label{eq:hyp-indext}
&\tilde \nu_{n ,\bi}(y)=
\|D_{T^n_{\bi}y}(\widetilde T_{\bi}^{n})^{-1} \|\, , 
\, \, y\in \tO_{\bi} \, , \,\, \bi \in I^n; \\ 
 \nonumber &\tilde \nu_n(x)=\sup_{\bi\in I^n: x \in {\tO}_{\bi}}\tilde \nu_{n,\bi}(x) \in [ \nu_n(x), \lambda^{-n }]\, , \, \, x\in M\, .
\end{align}
The function $\nu_n$   is  submultiplicative (multiplicative if $d=1$). We set
\begin{equation}
\nu_*(x) = \lim_{n \to \infty} \nu_n(x)^{1/n}\, , \, x \in M \setminus \SS_\OO \, .
\end{equation}

For $1<p<\infty$ and $t\ge 0$, we denote
by $\HH^t_p=\HH^t_p(M)$   the standard Sobolev space on $M$ (see Section \ref{sec:sobolev}).
We write $r_{\text{ess}}(\LL|_{\BB})$ for
the essential spectral radius of a bounded operator $\LL$ on a Banach space $\BB$.
For a fixed piecewise expanding map $T$, 
a function $f:M\to \complex$ is called piecewise 
continuous if  $f|_{O_i}$ extends  continuously to $\tO_i$,
and $f$ is called piecewise 
$\CC^\alpha$ if  $f|_{O_i}$ is  $\CC^\alpha$.
If $f$ is piecewise 
$\CC^\alpha$, it is easy to see
that each
 $f|_{O_i}$ admits a $\CC^\alpha$ extension $\tilde f_i$ (with\footnote{It is enough to consider real-valued $f$ in charts. Set $\tilde f_i(x)=\inf_{y \in O_i} \{f(y)+ C_i d(x,y)^\alpha\}$.} the same H\"older
constant $C_i$) to  $\tO_i$  for each $i\in \{1,..,I\}$, and we set
\begin{equation}
 \label{defnn} 
 f^{(n)}(x)=\prod_{k=0}^{n-1} f(T^{k}(x)) \, , x\in M\, ,\,\,
\tilde f^{(n)}_\bi=\prod_ {k=0}^{n-1} \tilde f_{i_k}(\widetilde T^k_\bi(x))\,,\, \bi\in I^n\, ,\,  x\in \tO_\bi \, .
\end{equation}

We can now define the transfer operator:

\begin{definition}
Let $T:M\to M$ be  piecewise
$C^{\bar\alpha}$ expanding. Let $g:M\to \bC$ be 
piecewise $\CC^\alpha$. Fix $1\le p\le \infty$. The transfer operator
 $\LL_{g}:L_{p}(M)\to L_{p}(M)$
is 
\[
\LL_{g}(\varphi)(x)=\LL_{T, g}(\varphi)(x)=\sum_{y: T(y)=x} g(y) \varphi(y), \quad x \in M\, ,\,  \, \varphi \in L_p(M)\, .
\]
\end{definition}

Next,  for $\mu \in \Erg(T)$
(the set of ergodic $T$-invariant probability measures on $M$),
 let $h_\mu=h_\mu(T)$ be its
 Kolmogorov entropy  and   $\chi_\mu(DT)$ the smallest Lyapunov
exponent of the linear cocycle $DT$.
If $d=1$,  the Birkhoff ergodic theorem gives
$\chi_\mu(DT)=\int \log |T'| \, d \mu=\int \log |\det DT| \, d \mu$.

\smallskip
Finally,
the \emph{asymptotic  complexity at the beginning  (or entropy multiplicity, see e.g. \cite{Buz97})
$D^b=h_{\mult}$}  of $T$  (and $\OO$) is
\begin{equation}\label{complexity-cond}
D^b=D^b(T)=\lim_{n\to \infty} \frac{1}{n}\log D_n^b(T) \, ,
\end{equation}
where the {\it $n$-complexity at the beginning  $D_{n}^{b}(T)$} of $T$ (and $\OO$) is 
\begin{equation}\label{complexity-condn}
D_n^b=D_{n}^{b}(T)=\max _{x \in M} \#\left\{\bi=\left(i_{0}, \ldots, i_{n-1}\right) \mid  \overline{O_{{\bi}}}\ni x\right\} 
 \, , \,\, n\ge 1 \,\, .
\end{equation}


\begin{remark}[Complexity at the end $D^e$]\label{Dend}
The works \cite{BaGo9,BaGo10,Tho} also use \emph{complexity at
the end} 
\begin{align*}&D_n^e=D_{n}^{e}(T)=\max _{x \in M} \#\{\bi=\left(i_{0}, \ldots, i_{n-1}\right) \mid x \in \overline{T^{n}\left(O_{\bi}\right)}\}\, ,\,\,
n\ge 1 \, , \\
 &D^e=D^e(T)=\lim_{n\to \infty}\frac 1 n \log D^e_n >0\, .
\end{align*}
For $Tx=2x \mod 1$ on $[0,1]$, we have $D^e_n(T)=2^n$, see  Remarks~\ref{somewhere} and~\ref{Dend2},
and Lemma~\ref{finally} for 
more about  complexity at the end. 
Thomine's bound
\cite{Tho} is
\begin{equation}\label{DDT}
r_{\text{ess}}(\LL_g|_{\HH^t_p})\le \lim_{n\to \infty} \biggl (D^b_n(T)^{\frac 1 p}\, D^e_n(T)^{\frac{p-1}p}
\, \sup \biggl |g^{(n)} |\det DT^n|^{\frac 1 p} \nu_n^t \biggr |\biggr ) ^{1/n}\, .
\end{equation}
Choosing $p>1$ close to $1$ allows to control the contribution of $D^e_n$, but
such an exponent $p$ increases the contribution of $|\det DT^n|^{1/p}$. 
 In our estimates, the complexity at the end  will be implicit in the topological pressure 
 $P^*_{\optop}$.
\end{remark}

\subsection{Pressure $P^*_{\optop}$ and Boundary Pressure.}

We define the  pressure of 
 subadditive  sequences  for a piecewise expanding map $T$,
generalising the pressure\footnote{The partition $\OO$ is a topological
generator, but $T$ is not continuous, so classical results do not
apply. We do  not relate here   $P^*_{\optop}$
to  pressure defined via open covers, separated sets, or
spanning sets. See \cite{BaDe1, BaDe2} for  entropy
and pressure via natural topological generators for billiards.}  $P^*_{\optop}(T,\log f, E)$ (for
$E\subset M$ and $f:M\to \real^+_*$) studied e.g. by Buzzi--Sarig \cite{Buz}.

\begin{definition}[Pressure of a Subadditive Potential]
A submultiplicative sequence
for the piecewise $\CC^{\bar\alpha}$ expanding
map $T$ is a sequence $\{f_n:M \to \real^+\mid n\ge 1\}$ of bounded
functions with   $f_{m+n}(x)\le f_{m}(T^{n}( x)) \cdot f_n(x)$ for all $m, n \ge 1$.
For $E\subset M$ measurable, and  $\{ f_n\}$  submultiplicative, the {\it topological pressure of $T$ and 
(the subadditive potential) $\{\log f_n\mid n \ge 1\}$ on $E$} is\footnote{The limit exists in $\real \cup \{-\infty\}$ by subadditivity.}
\begin{equation}
P^*_{\optop}(T,\{\log f_n\mid n \ge 1\}, E)=\lim_{n\to \infty} 
\frac 1n \log 
\sum_{\substack{\bi \in I^n:}\\{E\cap \bar O_{\bi}\neq \emptyset}}
\sup _{\overline{O_{\bi}}} f_{n} \in [-\infty, \infty)\, .
\end{equation}
\end{definition}

 We write  $P^*_{\optop}(\{ \log f_n\},E)$ and $P^*_{\optop}(\{\log f_n\})$
when the meaning is clear.
If $f_n=f^{(n)}$ is multiplicative, we just write  $P^*_{\optop}(\log f,E)$  and $P^*_{\optop}(\log f)$.
The {\it topological entropy of $T$ on a measurable set $E\subset M$} is $P^*_{\optop}(T,0, E)$. 

\smallskip
For all $q\ge 1$, we  have the trivial bound
\begin{align}
\label{forTh0}\E^{\frac{P^*_{\optop}(0,E)}q}
\cdot \lim_{n\to \infty} (\inf  f_n)^{1/n}
\le \E^{\frac{P^*_{\optop}(\{q\log f_{n}\},E)}q}&\le
\E^{\frac{P^*_{\optop}(0,E)}q}
\cdot \lim_{n\to \infty} (\sup  f_n)^{1/n} 
\, .
\end{align}

\begin{remark}[Comparing Pressure with Complexity]\label{somewhere}
Note that $e^{D^e(T)}$ is just the spectral radius of $\LL_1$ on $L_\infty$: Indeed, we have
\begin{equation}
\label{notenough}
\max \LL^n_1(1) = e^{D^e_n(T)}  \le \exp\#\{ \bi \in I^n \mid O_\bi \ne \emptyset\}\, .\mbox{ Thus, } D^e(T)\le P^*_{\optop}(0)\, .
\end{equation}
For the complexity at the beginning, we have
\begin{equation}\label{checkagain}
D^b(T)\le P^*_{\optop}(0,\partial \OO)\, ,
\end{equation}
(Indeed, setting $P^{*}_{n,\optop}(0,\partial \OO)=\#\{\bi \in I^n \mid \partial \OO\cap \bar O_{\bi}\neq \emptyset\}$, we have, 
$
D^b_1(T)\le  P^{*}_{1,\optop}(0,\partial \OO)$ and 
$D_n^b(T)\le \max \{D_{n-1}^b(T), P^{*}_{n,\optop}(0,\partial \OO)\}$ if $n\ge 2$.)

In the other direction, using that $\partial \OO$ has codimension
 one, we have by \cite[Prop.~5.2]{Buz1} (condition (A2) there is satisfied) that
 \begin{equation}\label{Buzbd}
 P^*_{\optop}(-\log |\det DT|,\partial \OO) \le -\log \lambda + D^b(T)\, .
\end{equation}
Set $\Lambda_0=1$, and,\footnote{$\wedge^{k}(A)|_V$ denotes the quantity by which the $k$-dimensional 
volume of $V$ is multiplied by the linear map
$A$.} for $d\ge 2$,
\begin{equation}\label{defLambda}
\Lambda_{d-1}=\exp \limsup_{n \to \infty}
\frac 1 n \log \max_{x, V} \|\wedge^{d-1}(DT^n_x)|_V\|\, ,
\end{equation}
 where the maximum ranges over $(d-1)$-dimensional  subspaces $V$ of $\TT_x M$. If $T$ is piecewise affine, then
\cite[Prop.~4]{Buz97} 
implies that
 \begin{align}
\nonumber  P^*_{\optop}(\log f,\partial \OO) &\le \sup \log f+ P^*_{\optop}(0,\partial \OO)\\
\label{Buzbd2} &\le \sup \log f + \log \Lambda_{d-1} + D^b(T)\, .
\end{align}
\end{remark}

\begin{definition}[Small Boundary Pressure]
Let $T$ be piecewise $\CC^{\bar\alpha}$ expanding. A  submultiplicative sequence $\{f_n:M \to \real^+\mid n\ge 1\}$,
  satisfies the {\it small boundary pressure condition}  if 
\begin{equation}\label{eq:small-boundary}
P^*_{\optop}(T,\{\log f_n\}, \partial \OO)<P^*_{\optop}(T,\{\log f_n\})\, .
\end{equation}
\end{definition}

For  {\it multiplicative sequences} $f_n=f^{(n)}$, associated
to a  piecewise
$\CC^\alpha$ function $f:M\to \real^+$ with $\inf f>0$,   Buzzi and Sarig  (\cite[Thm~1.3]{Buz}) showed
\begin{align}
\nonumber P^*_{\optop}(T,\log f, \partial \OO)&<P^*_{\optop}(T,\log f)
\Longrightarrow \\
\label{laref}&P^*_{\optop}(\log f)=
\sup_{\mu\in \Erg(T)} \bigl\{h_\mu(T)+\int \log f \D\mu\bigr\}
\, .
\end{align}
They also showed that small boundary pressure implies that there are finitely many measures
realising the supremum, and that, if $T$ is strongly topologically
mixing,  this maximum is uniquely attained. In a previous work, Buzzi \cite[Thm~A]{Buz1} had 
established \eqref{laref} for $g=|\det DT|^{-1}$, showing also that
\begin{align}
\nonumber P^*_{\optop}(T,-\log|\det DT| , \partial \OO)<P^*_{\optop}(T,&-\log |\det DT|)
\Longrightarrow \\
\label{larefB}&P^*_{\optop}(-\log |\det DT|)=0
\, .
\end{align}

Theorem~\ref{thm:var-princ} below generalises  \eqref{laref}  to  certain subadditive potentials. 

A useful consequence of \eqref{laref} is the following lemma:

\begin{lemma}[Small Boundary Pressure Implies $P^*_{\optop}(T,0)\le D^e(T)$]\label{finally}
Let $T$ be piecewise $\CC^{\bar\alpha}$ expanding,  let $g:M\to \real$ be piecewise $\CC^\alpha$
with $\inf g>0$.
If  $P^*_{\optop}(T,\log g, \partial \OO)<P^*_{\optop}(T,\log g)$, then 
$P^*_{\optop}(T,\log g)$ is bounded by the logarithm of the spectral radius of $\LL_g$ on
$L_\infty$. (In particular, if  $P^*_{\optop}(T,0, \partial \OO)<P^*_{\optop}(T,0)$
 then $D^e(T)=P^*_{\optop}(T,0)$ by the equality in \eqref{notenough}.)
\end{lemma}

The fact that $P^*_{\optop}(T,\log g)$ is bounded by the logarithm of the spectral radius of $\LL_g$ on
$L_\infty$ for $g>0$ is well known if $d=1$ (see \cite[Thm~3.3]{Babook1}).

\begin{proof}
In view of \eqref{laref}, for any $\delta >0$ there exists $\mu_\delta \in \Erg(T)$
such that $P^*_{\optop}(T,\log g)\le h_{\mu_\delta}(T) +\int \log g \D \mu_\delta+ \delta$.
The claim  thus follows from an application of Rohlin's formula.
See e.g. \cite[pp.~160--161]{Babook1}.
\end{proof}


\subsection{The Essential Spectral Radius (Theorem~\ref{thm:ess-bound} and Corollary~\ref{cor1})}

We introduce
the   {\it weighted  $n$-complexity at the beginning  $D_{n}^{b}(T,\bar f)$} of $T$
and a nonnegative function $\bar f$:
\begin{equation}
D_n^b(\log \bar f)=D_{n}^{b}(T,\log \bar f)=\sup _{x \in M} 
\sum_{\bi \in I^n\mid \overline{O_{\bi}}\ni x} \sup_{\overline{O_{\bi}}}
\bar f
 \, , \,\, n\ge 1 \,\, ,
\end{equation}
and we set, for a submultiplicative
sequence of nonnegative functions $f_n$, \begin{equation}\label{complexity-condw}
D^b(\{\log f_n\})=D^b(T,\{\log f_n\})=\lim_{n\to \infty} \frac{1}{n}\log D_n^b(T,\log f_n) \, .
\end{equation}
If $f_n\equiv 1$ for all $n$, we recover $D_n^b(T)$ and $D^b(T)$ from \eqref{complexity-condn},
\eqref{complexity-cond}.

\smallskip

 We have the following generalisation of
 \eqref{checkagain}:
  \begin{lemma}\label{boundary}
We have $D^b(T, \{\log f_n\})\le P_{\optop}^* ( \{\log f_n\}, \partial \OO)$.
\end{lemma}
\begin{proof}
Setting $P^{*}_{n,\optop}( \log \bar f ,\partial \OO)=\sum_{\bi \in I^n \mid \partial \OO\cap \bar O_{\bi}\neq \emptyset}
 \sup_{\overline{O_{\bi}}} \bar f $, we have 
$$
D^b_1(T, \log \bar f)\le  P^{*}_{1,\optop}( \log \bar f,\partial \OO)$$
 and   (recalling also that $f_{n}\le (f_1 \circ T^{n-1})\cdot f_{n-1}$),
$$D_n^b(T, \log f_n)\le \max \{D_{n-1}^b(T,\log f_n), P^{*}_{n,\optop}( \log f_n,\partial \OO)\}
\, , \, \, \forall n\ge 2\, . \qedhere
$$
\end{proof}
The first main result  follows. (It  is proved in \S\ref{sec:proof-thm}.)

\begin{thm}[Spectral and Essential Spectral Radius]\label{thm:ess-bound}
Let $T:M\to M$ be piecewise $\CC^{\bar\alpha}$ expanding and recall  $\nu_n$  from   \eqref{eq:hyp-index}.
Let $g:M\to \complex$ be  piecewise $\CC^\alpha$. 
For all $p\in(1,\infty)$ and $t\in(0,\min\{1/p,\alpha\})$, the operator $\LL_g$ 
on $L_p(M)$ restricts  boundedly\footnote{See \S\ref{sec:sobolev} for the definition of $\HH^t_p(M)$.} to  $\HH^{t}_p(M)$,
with  essential spectral radius there bounded by 
\begin{equation}\label{eq:ess-bound1}
 R^{t,p}_*(g)= \exp\bigl( 
\frac{ D^b(T)}{p} +\frac {p-1} {p}  P^*_{\optop}(\{\frac p {p-1} \log \bigl(|g^{(n)}|\cdot |\det DT^n|^{\frac 1p}
\cdot \nu_n^{t}\bigr )\})\bigr )  .
\end{equation}
Moreover, for $s\in [0, t]$,  the essential spectral radius of $\LL_g|_{\HH^t_p(M)}$ is bounded
by\footnote{The term $D^b(T, \{\log \nu_n^{s}\})/p$ in   \eqref{eq:ess-bound1t'} can  be replaced
by $P^*_{\optop}( \{\log \nu_n^{s}\},\partial O)/p$,applying Lemma~\ref{boundary} to $f_n=\nu_n^s$.}
\begin{align}\label{eq:ess-bound1t'}
 R^{t,p}_{*,s}(g)=& \exp\biggl( 
\frac{ D^b(T, \{\log \nu_n^{s}\})}{p} \\
\nonumber &\qquad+\frac {p-1} {p}  P^*_{\optop}(\{\frac p {p-1} \log (|g^{(n)}|\cdot |\det DT^n|^{\frac 1p}
\cdot \nu_n^{t-s} )\})\biggr ) \, .
\end{align}
Finally, the spectral radius of $\LL_g$ on $\HH^0_p(M)=L_p(M)$
is bounded by
$
 R^{0,p}_*(g)$.
\end{thm}

Note that by \eqref{forTh0}, we have (similar bounds can be written for $R^{t,p}_{*, s}(g)$)
\begin{align} 
&R^{t,p}_*(g)\le \nonumber \E^{ \frac { D^b(T)}p+ \frac { p-1} p P^*_{\optop}(\frac p {p-1} \log |g|)}
  \lim_{n\to \infty} \| |\det DT^n|^{\frac 1p} \nu_n^{t}\|_{L_\infty}^{\frac 1 n}\\
\label{eq:ess-bound1'}
&\qquad\qquad \le \E^{ \frac { D^b(T)} p+ \frac { p-1} p P^*_{\optop}(0)}
  \lim_{n\to \infty} \|g^{(n)}\cdot |\det DT^n|^{\frac 1p} \nu_n^{t}\|_{L_\infty}^{\frac 1 n} .
\end{align}

The Ruelle inequality is the property that 
\begin{equation}\label{Ri}\sup_{\mu \in \Erg(T)} \{h_{\mu}(T)- \int \log |\det DT| \D\mu\} \le 0\, . 
\end{equation}
By \eqref{laref} and \eqref{larefB},   small boundary
pressure 
for $f_n=|\det DT^n|^{-1}$ implies the Ruelle inequality. See \cite[(1)]{APu} for a ``large image'' condition
(called quasi-Markovianity there) which
ensures the Ruelle inequality.

We state a corollary of
\eqref{eq:ess-bound1}  (the reader is invited to derive   variational bounds
 from \eqref{eq:ess-bound1t'}, Lemma~\ref{boundary}):

\begin{cor}\label{cor1} In the setting of Theorem~\ref{thm:ess-bound}, set $f=(|g| |\det DT|^{1/p})^{p/(p-1)}$. Assume that the Ruelle inequality holds. If 
$P^*_{\optop}(T,\log f, \partial \OO)<P^*_{\optop}(T,\log f)$
then
\begin{align}
 R_*^{t, p}(g)&\le \nu_*^t \cdot \E^{ \frac { D^b(T)}p} 
\exp\bigl(\sup_{\mu\in \Erg(T)}\bigl\{ \frac {p-1}{p}h_\mu(T)+\int \log( |g|  |\det DT|^{1/p})\D\mu\bigr\}\bigr)\\
\nonumber &\le \nu_*^t \cdot \E^{ \frac { D^b(T)}p} 
\exp\bigl (\sup_{\mu\in \Erg(T)}\bigl\{\int \log( |g|  |\det DT|)\D\mu\bigr\}\bigr ) \, .
\end{align}
\end{cor}

Clearly,
$
\sup_{\mu\in \Erg(T)} 
\{\int \log( |g|  |\det DT|)\D\mu\}
\le \lim_{n\to \infty}  ( \|g^{(n)} |\det DT^n|\|_{L_\infty} )^{1/n}
$.

\begin{proof}[Proof of Corollary~\ref{cor1}]
The first bound follows from \eqref{laref} and \eqref{eq:ess-bound1}. To show the second
one, use 
Ruelle's inequality and
 proceed as for  \cite[(2.27)]{Babook2}.
\end{proof}

We list some comments about the unconditional result
Theorem~\ref{thm:ess-bound}.

\begin{remark}[The case $d=1$]\label{d1}
If $d=1$,
then   $D^b(T)=0$, and $|\det DT^n|=|DT^n|=\nu_n^{-1}$ so that  \eqref{eq:ess-bound1'} is
\begin{align}
\label{for1} &\le \E^{ \frac { p-1} p P^*_{\optop}(0)} 
\lim_{n\to \infty} \,  \bigl \|  g^{(n)} |DT^n|^{\frac 1 p -t} \bigr \|^{1/n}\, .
\end{align}
Letting $p\to 1$, $t\to 1$ in \eqref{for1} (or in \eqref{DDT}), we recover the
(sometimes optimal \cite{Kel}) bound
$\lim_n\sup |g^{(n)}|^{1/n}$ from
\cite{BaKe}  for the essential spectral radius  on $BV$.
\end{remark}

\begin{remark}[Spectral Gap if $g=|\det DT|^{-1}$]\label{compLiv}
For $g=|\det DT|^{-1}$ the dual of $\LL_g$ fixes Lebesgue measure, and Lebesgue measure belongs to the
dual of any $\HH^t_p(M)$ with $t\ge 0$.
In addition, the norm of $\|\LL_g\|_{L_1(M)}\le 1$,
with  $\HH^t_p(M)\subset L_1(M)$. 
Hence, if $r_{\text{ess}}(\LL_{|\det DT|^{-1}}|_{\HH^t_p})<1$,
for some $0<t<1/p$, then the spectral radius of $\LL_g$
on $\HH^t_p$ is equal to one and
 standard arguments imply that\footnote{See also the comment after
 \eqref{laref} for the existence of equilibrium states for general $g>0$ under small boundary pressure.}
$T$ has finitely many ergodic absolutely continuous
invariant probability measures, with densities
in $\HH^t_p(M)$, the union of whose ergodic basins has full measure (see e.g. \cite[Thm~33]{BaGo9} or \cite[Thm~1]{Liv13}).  Each of these measures is exponentially mixing
(up to  a finite period) for  $\CC^v$ H\"older observables if $v>t$.
  Taking  $t$ (and thus $p$) close enough to $1$, our bound
\eqref{eq:ess-bound1'} for $r_{\text{ess}}(\LL_{|\det DT|^{-1}}|_{\HH^t_p})$
is strictly smaller than one  if $D^b(T)=0$. More generally, 
if $\sigma=\exp({D^b(T, \{\nu_n^{s}\})})<1$ for some $s\le t$ then \eqref{eq:ess-bound1t'}
is bounded by $\sigma$ for $t$ and $p$ close enough to $1$, which is
comparable to Liverani's bound from \cite[(3), Lemma~3.1]{Liv13}.
\end{remark}

\begin{remark}[Spectral Gap if $g\ge 0$]
The spectral radius of $\LL_{T,g}$ on $L_\infty$
is 
$\le P^*_{\optop}(\log |g|)$ (similarly as for \eqref{notenough}).
Although $\HH^t_p$ is not included in $L_\infty$ if $t<1/p$,
 we conjecture,  in view of \cite[Thm~1.2]{Buz}, that,
if $g\ge 0$ and there exist  $t <\min \{1/p,\alpha\}$ such that
 $
r_{\text{ess}}(\LL_{g}|_{\HH^t_p})< P^*_{\optop}(\log g)=R_*^{0,\infty}(g)
$,
then the spectral radius of $\LL_{T,g}$ on $\HH^t_p$
is 
$P^*_{\optop}(\log g)$. In particular, combining the maximal 
eigenvectors of $\LL_g$ and its dual should then
give another construction for the equilibrium
states of Buzzi--Sarig \cite{Buz}, with the
additional perk of exponential decay of correlations for suitable observables.
\end{remark}

\begin{remark}[Comparing \eqref{eq:ess-bound1} with Thomine's bound \eqref{DDT}]\label{Dend2}
Our  bound \eqref{eq:ess-bound1'}  is  less than or equal to  
\eqref{DDT} as soon as $D^e(T)\ge P^*_{\optop}(0)$.
By Lemma
\ref{finally}, this holds under the small boundary entropy condition
$P^*_{\optop}(T,0,\partial \OO) < P^*_{\optop}(T, 0)$.
For $T$ a multidimensional $\beta$-transformation (i.e. a piecewise affine
$T$, with  $DT$ constant) on $[0,1]^d$, it is known \cite[Thm~1, Lemma~1]{Buz97} that $D^b(T)=0$ and  $P^*_{\optop}(0)=\htop=\sum_{i=1}^d \xi_i=\log|\det DT|$, for  $0<\xi_1\le \cdots \le \xi_d$ the Lyapunov exponents of $T$.
By \cite[Thm~1]{BuKe}, we have $D^e(T)=\lim_{n\to \infty} \max \LL_{T,1}^n (1)\ge \exp(\sum_{i=1}^d \xi_i)$.
Thus, 
we recover Thomine's bound \eqref{DDT} since \eqref{eq:ess-bound1'} gives
$$r_{\text{ess}}(\LL_{|\det DT|^{-1}}|_{\HH^t_p})\le \E^{-t \xi_1+\frac{p-1}{p}\sum_{i=1}^d \xi_i+\frac{1-p}{p}\sum_{i=1}^d \xi_i}=\E^{-t \xi_1}\, .
$$

Our bound \eqref{eq:ess-bound1} 
can be strictly smaller than \eqref{DDT}
(we expect that this holds generically if $D^e(T)= P^*_{\optop}(0)$): Take 
$d=1$ and\footnote{Then $P^*_{\optop}(0)=D^e(T)$.
Moreover,  $P^*_{\optop}(0)=\htop(T)>0$:  
Taking the $O_i$ to be maximal
monotonicity intervals, $\#\{ \bi \in I^n \mid O_\bi \ne \emptyset\}$ is
  the lap-number $lap_n$
of $T$, and   $\lim_{n\to \infty} n^{-1} \log lap_n=\htop(T)$.}  $T$  continuous, with $g= |\det DT|^{-1}$. By the variational principle \cite[Thm~3.1]{Buz} for $-\log |\det DT|$ (using \eqref{Buzbd} with $D^b(T)=0$), there exists $\musrb$ such that\footnote{Note that the left hand-side is equal to zero.}
\begin{align*}
\frac{p-1} p P^*_{\optop} ( \frac p {p-1} \log g&+ \frac 1 {p-1} \log |\det DT|)\\
&= \frac{p-1} p\bigl ( h_{\musrb} -\int \log |\det DT| d\musrb\bigr ) \\
&\le  \frac{p-1} p\bigl (P^*_{\optop} ( 0) -\int \log |\det DT| d\musrb \bigr ) \, . 
\end{align*}
If  $T$ has a fixed point with Lyapunov
exponent $<\int \log |\det DT| d\musrb$ then 
$$
\int \log |\det DT|^{-1+1/p} d\musrb
< \lim_n  \frac 1 n \log
\sup |\det DT^n|^{-1+1/p} \, .
$$
\end{remark}

\begin{remark}[Comparing \eqref{eq:ess-bound1} with Gundlach--Latushkin]\label{GL}
If  $T$ is $\CC^{\bar\alpha}$ and $g$ is $\CC^\alpha$
on $M$, the condition
 $t <1/p$ can be lifted, and  we get the optimal bound 
$$ \exp
\sup_{\mu\in \Erg(T)}\biggl(  h_\mu(T)+ \int \log |g| \D\mu-t\chi_{\mu}(DT)\biggr)
$$
from \cite{GuLa} for $\BB=\CC^\alpha$ 
by\footnote{In the smooth case $D^b(T)=0$,
and the variational principle \eqref{laref} for the subadditive potential
$\log \bigl(|g^{(n)}|\cdot |\det DT^n|^{\frac 1p}
\cdot \nu_n^{t}\bigr )$
holds  \cite[App.~B]{Babook2}.} letting
$p\to \infty$  in \eqref{eq:ess-bound1}.
\end{remark}

\subsection{Variational Principle (Theorem~\ref{thm:var-princ}
and Corollary~\ref{lecor})}

Our second main result (proved in Section \ref{sec:pressure}) is a variational principle for  certain subadditive potentials,  generalising \eqref{laref}:

\begin{thm}[Variational Principle]\label{thm:var-princ}
Let $T$ be piecewise $\CC^{\bar\alpha}$ expanding, let\footnote{Our application is $G=(g\cdot |\det DT|^{\frac 1p})^{q}$.} $G:M \to \complex$ be piecewise $\CC^\alpha$, and let $t\ge 0$.
If the small boundary pressure condition
 \eqref{eq:small-boundary}
holds for  $f_n=f_{n,t}=|G^{(n)}|\cdot \nu_n^t$ (recall \eqref{eq:hyp-index}), then 
\[
\sup_{\mu\in \Erg(T)}\bigl \{h_{\mu}(T)+\int \log |G|\D \mu-t\chi_{\mu}(D T)
\bigr \}= P^*_{\optop}(\{\log f_{n}\}) \, .
\]
In addition, the supremum above is attained.
\end{thm}

Define  
\begin{equation}\label{sequence1}
 f_{n,t,p}= |g^{(n)}|\cdot |\det DT^n|^{\frac 1p}
\cdot \nu_n^{t}\,,\, t\ge 0\, , \,  p \ge1\, , \, \, n\ge 1 \, .
\end{equation}

Theorem~\ref{thm:var-princ} and the proof of Theorem~\ref{thm:ess-bound}
allow us to show the following corollary in \S\ref{sec:proof-thm} see the proof of Corollary~\ref{cor1}
for the last claim):

\begin{cor}[Variational Expression for the Essential Spectral Radius]\label{lecor}
Let $T$, $g$, $p$, $t$, be as in Theorem~\ref{thm:ess-bound}.
Assume that the Ruelle inequality \eqref{Ri} holds. If    $T$ satisfies 
the small boundary pressure condition \eqref{eq:small-boundary}
 for  $f_n=f_{n,t,p}^{q}$ with $q\in[ 1, \frac p {p-1}]$, and $0\le t<\min \{1/p, \alpha\}$,  
then the spectral radius of $\LL_g$ on $\HH^0_p=L_p$ is bounded by
$R^{0,p,q}(g)$, the essential spectral radius of $\LL_{g}$ on ${\HH^t_p}$ is bounded
by 
\begin{align}
\nonumber R^{t,p,q}_*(g)= \exp& \biggl (\frac {D^b(T)}p \\
\label{eq:ess-bound2} &+
\sup_{\mu\in \Erg(T)}\bigl\{  \frac{ h_\mu(T)}{q}+ \int \log(|g||\det DT|^{\frac 1p}) \D\mu-t\chi_{\mu}(DT)
\bigr \}\biggr)\, .
\end{align}
In addition, we have
\begin{align}
\label{XY} R^{t,p,q}_*(g)\le \exp \bigl (\frac {D^b(T)}p\bigr ) \cdot \lim_{n\to \infty} & 
\bigl 
 \| g^{(n)}|\det DT^{(n)}|^{1/p+1/q}\, \nu_n^t\bigr \|_{L_\infty} ^{1/n}\, .
\end{align}
\end{cor}

If $D^b(T)=0$ and \eqref{eq:small-boundary}  holds for $f^q_{n,t,p}$ with $q=p/(p-1)$, the bound \eqref{eq:ess-bound2} 
reduces to the bound in \cite[Thm~2.15]{Babook2} for smooth expanding maps.

Note that $1/p+1/q\in [1, 1+1/p]$.
If $q=p/(p-1)$ then $1/p+1/q=1$, and, comparing \eqref{XY} to
Thomine's bound \eqref{DDT},  we see that the complexity at the end has disappeared, at the cost
of a higher weight on $|\det DT|$.

\begin{remark}[Small Boundary Pressure Condition \eqref{eq:small-boundary}]\label{condest}
 The bound \eqref{eq:small-boundary}   
for  $f_{n,t,p}^{q}$ holds\footnote{Cf. the condition $\sup \log |g|- \inf \log |g|<P^*(0)$ from the pioneering work \cite{HoKe}.}  if
 \begin{align}\nonumber
q\bigl (\log [\sup  (|g||\det DT|^{1/p}\lim_n (\sup \nu_n^{t})^{1/n})]
&-
\log \bigl[ \inf ( |g||\det DT|^{1/p})\lim_n (\inf \nu_n^{t})^{1/n}\bigr]\bigr )
\\
\label{lacond} &\qquad\qquad\qquad<{P^*_{\optop}(0)}-P^*_{\optop}(0, \partial \OO)
 \, .
\end{align}
Indeed,  by \eqref{forTh0}, we find
$$
\E^{\frac{P^*_{\optop}(\{q\log f_{n,t,p}\})}q}\ge \E^{\frac {P^*_{\optop}(0)}q}
\lim_{n \to \infty} \inf  [ |g^{(n)}|\cdot |\det DT^n|^{\frac 1p}\nu_n^{t}] ^{1/n}\, ,
$$
while
 \eqref{forTh0}  gives
\begin{align*}
\E^{\frac{P^*_{\optop}(\{q\log f_{n,t,p}\}, \partial \OO)}q}&\le \E^{\frac{P^*_{\optop}(0, \partial \OO)}q}  \lim_{n \to \infty}\sup  [ |g^{(n)}|\cdot |\det DT^n|^{\frac 1p}\nu_n^{t}] ^{1/n}
\, .
\end{align*}
In the piecewise affine case,  \eqref{Buzbd2} gives
$P^*_{\optop}(0, \partial \OO)\le {D^b(T)+\log \Lambda_{d-1}}$ (recall \eqref{defLambda}).
 If $d\ge 2$,
  \eqref{lacond} holds for $\beta$-transformations $T$ and $g=|\det DT|^{-1}$ with  $0<t<1/p$, for
arbitrary $q\ge 1$.
Indeed, $D^b(T)=0$ and $P^*_{\optop}(0)= \sum_{i=1}^d \xi_i\equiv\log |\det DT|$,
with $\nu_n$ a constant function, and $\lim_n \nu_n^{1/n}=\exp(-\xi_1)$,
so that   \eqref{lacond} reads
$
{ \sum_{i=2}^d \xi_i} <{ \sum_{i=1}^d \xi_i} 
$.
\end{remark}

\section{Proof of Theorem~\ref{thm:ess-bound} on the Essential Spectral Radius}\label{sec:PartI}
\subsection{Sobolev Spaces}\label{sec:sobolev}
For $p\in (1,\infty)$ and $t\in \real$,  define local Sobolev spaces by
$$H^t_p=H^t_p(\bR^d)=\{u\in L_p(\bR^d)\mid \|u\|_{H_{p}^{t}(\bR^d)}=\|\bF^{-1}((1+|\xi|^2)^{t/2}\cdot \bF u)\|_{L_p(\bR^d)} <\infty\}
\, ,$$
 where $\bF$ is the Fourier transform. 
If $t\ge 0$ then $H^t_p(\bR^d)$
is the closure of the Schwartz
 space $\SS$ of rapidly decreasing functions for the norm $\|u\|_{H_{p}^{t}(\bR^d)}$,
see e.g. \cite[Thm~3.2/2, Rk~3.2/2]{Tr}.

Recall $\widetilde M$ from the beginning of \S\ref{sec:def}.
To patch the local spaces together, we will
use charts and partitions of unity:

\begin{definition}[Admissible Charts and Partition of Unity
for $T$]\label{ChP}
For a piecewise $\CC^{\bar\alpha}$
expanding map $T$, we define admissible charts  to be 
a finite system of $\CC^{\infty}$ local charts $\{(V_\omega, \kappa_\omega)\}_{\omega\in \Omega}$, where
each
$V_\omega$  is open in $\widetilde M$ and such that
$\overline V_\omega \subset \tO_{i(\omega)}$, for some $i(\omega)\in I$, 
with  $M \subset \cup_\omega V_\omega$, and where each
$\kappa_\omega : V_\omega\to\bR^d$ is a  diffeomorphism
onto its image.
We define an admissible  partition
of unity to be a  $\CC^{\infty}$ partition of  unity  $\{\theta_\omega:
\widetilde M\to [0,1]\}_{\omega\in \Omega}$
   such that the support of $\theta_\omega$ is contained
in  $V_\omega$.
 \end{definition}

Note that if $M$ has a boundary,  the boundary
	cutoff will be performed through the characteristic functions
	of the $O_i$.

\begin{definition}[$\HH^t_p(M)$]
Let $T$ be piecewise $\CC^{\bar\alpha}$
expanding and let $\kappa_{\omega}$ and $\theta_{\omega}$ be as in  Definition~\ref{ChP}. For  $1<p<\infty$ and
$t<1/p$,  let $\HH^t_p=\HH^t_p(M)$ be defined by
$$
\HH^t_p(M)=\{\varphi \in L_p(M)\mid \|\varphi\|_{H_{p}^{t}(M)}=\sum _{{\omega} \in {\Omega}}\left\|\left(\theta_{{\omega}} \cdot \varphi\right) \circ \kappa_{{\omega}}^{-1}\right\|_{H_{p}^{t}(\bR^d)}<\infty \} . $$
\end{definition}

By \cite{Str}, the definition above makes sense if $M$ has a boundary (in that case
 $\kappa_\omega(\supp(\theta_\omega))$ has a boundary for some
$\omega$) since $t<1/p$.
Changing the system of charts or the partition of unity  produces  equivalent norms (see \cite[I.5, I.6]{Tay}).
It is easy to see that
$\HH^t_p(M)$ is the closure of  $\CC^v(M)$
 for the norm $\|\varphi\|_{H^t_p(M)}$ for  $v>t$.

\subsection{Toolbox. Zoomed Norms}\label{sec:box}We collect  results used for the  Lasota--Yorke inequality.
We shall use   the following localisation
to zoom into  smaller scales.
\begin{lemma}[Localisation {\cite[Thm~2.4.7(ii)]{Tri}}]\label{lem:loc-prin}
Fix $\tau$ in the set $\CC^\infty_0(\bR^d,[0,1])$ of compactly supported $C^\infty$ functions
from $\real^d$ to $[0,1]$.
For  $x\in\bR^d$ and $m\in \bZ^d$, set  $\tau_m(x)=\tau(x+m)$. For any $p\in (1,\infty)$ and $t\in \bR^+$ 
there exists  $C_{t,p,\tau}<\infty$ such that
\begin{equation}\label{eq:localization}
\bigg(\sum_{m \in \integer^{d}}\|\tau_{m} u\|_{H_{p}^{t}}^{p}\bigg)^{\frac{1}{p}} \leq C_{t,p,\tau}\|u\|_{H_{p}^{t}} \, , \, \forall u \in H^t_p (\real^d)\, .
\end{equation}
In addition,  if $\sum_{m\in \bZ^d}\tau_m(x)=1$ for all $x$, then there exists $C_{t,p,\tau}<\infty$ such that
\begin{equation}
\|u\|_{H^t_p}\le C_{t,p,\tau}\bigg(\sum_{m \in \integer^{d}}\|\tau_{m} u\|_{H_{p}^{t}}^{p}\bigg)^{\frac{1}{p}}
\, , \,\forall u \mbox{ such that: } \forall m\, ,\, \tau_m u\in H^t_p(\real^d)\, .
\end{equation}
\end{lemma}

For any measurable set $O$ and every
$p\in (1,\infty)$ we have $\|\mathbf{1}_{O} \varphi\|_{L_{p}} \leq \|\varphi\|_{L_{p}}$.
The next result   is the reason behind the constraint $t<1/p$:
\begin{lemma}[Characteristic Functions as Bounded Multipliers  {\cite[Cor.~II.4.2]{Str}}]\label{lem:charact}
For any  $-1+1/p< t<1/p<1$ there exists $C_{t,p}$  such that, for any $L \ge 1$,
and every measurable set $O \subset \bR^d$ whose intersection with almost every line parallel to some coordinate axis has at most $L$ connected components, 
we have
\[
\|\mathbf{1}_{O} \varphi\|_{H_{p}^{t}} \leq C_{t, p} L\|\varphi\|_{H_{p}^{t}}\, ,\, \forall \varphi \in H^t_p(\real^d)\, .
\]
\end{lemma}

We will be able to use Lemma~\ref{lem:charact} when composing  with the iterate $T^n$ of a piecewise expanding
map in view of  the following result. 

\begin{lemma} [{\cite[Lemma~5.1]{Tho}}]\label{lem:boundary-comp}
Recall Definitions~\ref{object} and \ref{ChP}. Let $L_0$ be the maximal number of smooth boundary components of the $O_i$. For any $n \ge 1$, $\bi\in I^n$, $x\in \overline O_{\bi}$, and for any $\omega \in \Omega$ such that $x\in \supp \theta_\omega$, there exist a neighbourhood $O'$ of $x$ and an orthogonal matrix $A$ such that the intersection of $A(\kappa_\omega (O' \cap O_{\bi}))$ with almost any line parallel to a coordinate axis has at most $L_0 n$ components.
\end{lemma}

\smallskip

The next result is crucial (for $F=\id$, it  plays the part of a Leibniz bound):

\begin{lemma}[Local Lasota--Yorke Bound {\cite[Lemma~2.21]{Babook2}}]
\label{lem:LY-local} Let $d\ge 1$. For each $0\le t <\alpha $ there exists  $c_t$ such that 
 for any $p\in (1,\infty)$, any open $U\subset \bR^d$ any $F:U\to \bR^d$ extending to a bilipschitz $\CC^{\bar\alpha}$ diffeomorphism of $\real^{d}$
 with 
\begin{equation}\label{control}
\sup _{\real^{d}}|\det D F| \le 2 \sup _{U}|\det D F|\, ,
\end{equation}
 and any  $\CC^\alpha$ function
 $f:\bR^d\to \bC$  supported in a compact set $K\subset U$, we have 
\[
\|f \cdot(\varphi \circ F)\|_{H_{p}^{0}} \leq  \sup _{K}|f| \sup _{U}|\det D F|^{-1 / p}\|\varphi\|_{H_{p}^{0}}\, ,
\, \forall \varphi \in H^0_p(\bR^d)=L_p(\real^d)\, ,
\]
and for any $t'<t$ there is $C_{t,t',p}(f,F)$ such that, for any  $\varphi\in \HH^t_p$ supported in $K$,
\[
\|f \cdot(\varphi \circ F)\|_{H_{p}^{t}} \leq c_{t} \sup _{K}|f| \sup _{U}\|D F\|^{t} \sup _{U}|\det D F|^{-1 / p}\|\varphi\|_{H_{p}^{t}}
+C_{t,t',p}(f,F)\|\varphi\|_{H_{p}^{t'}}\, .
\]
\end{lemma}

The intersection multiplicity of a family of subsets of $\bR^d$  is the maximal number of sets having nonempty intersection. The intersection multiplicity of a partition of unity 
of $\bR^d$ is the intersection multiplicity of the family formed by taking the supports of the maps in the 
partition.
With this terminology, we recall the {\it fragmentation-reconstitution} lemma, at the core of our local computations.

\begin{lemma}[Fragmentation and Reconstitution
{\cite[Lemmas 2.26 and 2.27]{Babook2}}]\label{lem:frag-rec}
Let $1<p<\infty$ and $t\ge 0$, and let $K\subset \bR^d$ be compact. For any $t'\in \bZ$ there exists a constant $C>0$ such that, for any partition of unity $\{\theta_j\}_{j=1}^J$ of $K$ with intersection multiplicity $\beta$, there
 exist finite constants $C_\theta,\tilde C_{\theta}>0$ (depending on the $\theta_j$ only through their supports) such that
\begin{align}
\label{frag}
&\|\sum_{j=1}^J \theta_j w\|_{H_p^t} 
\leq \beta^{\frac{p-1}{p}}(\sum_{j=1}^J \left\|\theta_j w\right\|_{H_p^t}^p)^{1/p}
+\tilde C_{\theta} \sum_{j=1}^{J}\left\|\theta_j w\right\|_{H_p^{t'}}\,
\mbox{(fragmentation),}\\
\label{rec}
& (\sum_{j=1}^J\left\|\theta_j w\right\|_{H_p^t}^p)^{1/p}
 \leq C \beta^{1 / p}\sup_{1\le j\le J} \|\theta_jw\|_{H_p^t}
 +C_\theta\sum_{j=1}^J\|\theta_jw\|_{H_p^{t'}}\,  \mbox{(reconstitution).} 
\end{align}
In addition, if $t=0$, then\footnote{This follows from the H\"older inequality. See the proof of \cite[Lemmas 2.26 and 2.27]{Babook2}.} we may take $C_\theta=\tilde C_\theta=0$.
\end{lemma}

Finally, for $p\in(1,\infty)$ and  $t \ge 0$, following Thomine \cite{Tho}, define a {\it zoomed norm}    for any increasing sequence $r_n>1$ (chosen in
the proof of  Proposition~\ref{lem:main-LY}) by
\begin{equation}\label{eq:zoom-norm}
R_n(x)=r_n \cdot x\, , \forall x \in \bR^d\, , \, n \in \integer_+\, ; \,\, 
\|\varphi\|_{r_n,t,p}=\sum_{\omega \in \Omega}\left\|\left(\theta_{\omega} \varphi\right) \circ \kappa_{\omega}^{-1} \circ R_{n}^{-1}\right\|_{H_{p}^{t}}\, .
\end{equation}
The zoomed norm $\|\varphi\|_{r_n,t,p}$ defined above is equivalent to $\|\cdot\|_{\HH^t_p}$.
It is used for example when applying Lemma~\ref{lem:boundary-comp} 
and Lemma~\ref{lem:LY-local} below.

\subsection{A Global Lasota--Yorke Inequality}\label{sec:LY}

We will prove the Lasota--Yorke inequality by 
combining  the zoomed norm  with the fragmentation-reconstitution techniques from
\S\ref{sec:box}, to obtain a thermodynamic factor in front of the strong norm.  
Recall the sequence $f_{n,t,p}= |g^{(n)}|\cdot |\det DT^n|^{\frac 1p}
\cdot \nu_n^{t}$  from \eqref{sequence1}.

\begin{prop}[Lasota--Yorke Bound]\label{lem:main-LY}
Fix  $p\in(1,\infty)$, and $0<t<\min\{\frac 1 p,\alpha\}$. Then there exists $C_{t,p}$
 such that, for any $T$ and $g$ as in Theorem~\ref{thm:ess-bound},
\begin{align}
\nonumber \|\mathcal{L}_{g}^{n} \varphi\|_{L_p} \leq C_{t,p} 
 (D_n^b)^{\frac{1}{p}} 
\bigl( \sum_{\bi\in  I^n} & \sup _{O_\bi} f^{\frac p {p-1}}_{n,0,p}\bigr)^{\frac{p-1}{p}}\cdot  \|\varphi\|_{L_p}
\, , \, \,\forall \varphi\in L_p(M)\, ,\,
 \, \forall n \ge 1\, ,
\end{align}
and, in addition, for each
$T$, there exists  an increasing\footnote{The sequence $r_n$ is independent of $t$ and $p$.} sequence $\{r_n=r_n(T)\}$
  such that,
for each 
 $g$   and each $t'<t$, there exists $C_n=C_{n,t,t',p}(g)$ 
 such that  
\begin{align}
\nonumber \|\mathcal{L}_{g}^{n} \varphi\|_{r_n,t,p} \leq C_{t,p} n 
 (D_n^b)^{\frac{1}{p}} 
\bigl( \sum_{\bi\in  I^n} & \sup _{O_\bi} f^{\frac p {p-1}}_{n,t,p}\bigr)^{\frac{p-1}{p}}\cdot  \|\varphi\|_{r_n,t,p}
\\
\label{eq:main-LY}&\qquad
+C_n\|\varphi\|_{r_n,t',p}\, , \, \,\,\forall \varphi\in \HH^t_p\, ,\,
 \, \forall n \ge 1\, .
\end{align}
\end{prop}

\begin{proof}
We prove \eqref{eq:main-LY}. (The bound on $L_p(M)$ follows from
a simplification of the argument for \eqref{eq:main-LY}, using the $L_p$ bound
in Lemma~\ref{lem:LY-local} and the last claim of Lemma~\ref{lem:frag-rec}. In particular,
the zoom is not needed.)
For $n\ge 1$ and each $\bi\in I^n$, select a   $\CC^\infty$ function
 $\theta_{\bi}$ on $M$ such that $\supp\theta_\bi \subset \widetilde O_\bi$ and $\theta_\bi\equiv 1$ on $O_\bi$.
Then we have\footnote{Since $\opleb(\partial \OO)=0$  and $\LL_g$ acts on a subset of
$L_p(M)$, the formula \eqref{eq:L-charact} is legitimate.}
\begin{equation}\label{eq:L-charact}
\mathcal{L}_{g}^{n} \varphi(x)=\sum_{\bi \in I^n}({\theta_\bi \tilde g_\bi^{(n)}}  \mathbf{1}_{O_{\bi}}\varphi) \circ \widetilde T_{\bi}^{-n}(x)\, , \mbox{ for Lebesgue a.e. } x\in M\, .
\end{equation}
Recall Definition~\ref{ChP}. For $\omega\in \Omega$,  we have 
\begin{equation}
\left(\theta_{\omega} \mathcal{L}_{g}^{n} \varphi\right) \circ \kappa_{\omega}^{-1} \circ R_{n}^{-1}
=\sum_{\bi\in I^n}
\biggl (\theta_{\omega}\bigl [( \theta_\bi{\tilde g_\bi^{(n)}}  \mathbf{1}_{O_{\bi}}\varphi ) \circ \widetilde T_{\bi}^{-n}
\bigr ]\biggr) \circ \kappa_{\omega}^{-1} \circ R_{n}^{-1}\, .
\end{equation}
Since $\varphi=\sum_{\omega' \in \Omega} \theta_{\omega'} \varphi$,
the triangle inequality followed by 
 the fragmentation bound \eqref{frag} in
Lemma~\ref{lem:frag-rec} applied for fixed $\bi$ to  the partition of unity $\{\theta_{\omega'}\}$ gives
constants $C$ (depending on the intersection multiplicity) and $\tilde C_{\{\theta_{\omega'}\}}$
such that
 \begin{align}
\nonumber
\|\left(\theta_{\omega} \mathcal{L}_{g}^{n} \varphi\right) &\circ \kappa_{\omega}^{-1} \circ R_{n}^{-1}\|_{H^t_p}\\
\label{eq:decomp-transfer2} \le & C\sum_{\bi\in I^n}\biggl (\sum_{\omega'}\|
\bigl ( \theta_{\omega}\bigl [( \theta_\bi{\tilde g_\bi^{(n)}}  \mathbf{1}_{O_{\bi}}\theta_{\omega'}\varphi ) \circ \widetilde T_{\bi}^{-n}\bigr ]\bigr ) \circ \kappa_{\omega}^{-1} \circ R_{n}^{-1}\|^p_{H^t_p}\biggr )^{1/p}
\\
\nonumber 
&+\tilde C_{\{\theta_{\omega'}\}}\sum_{\bi\in I^n}\sum_{\omega'}
\|\bigl ( \theta_{\omega}\bigl [( \theta_\bi{\tilde g_\bi^{(n)}}  \mathbf{1}_{O_{\bi}}\theta_{\omega'}\varphi ) \circ \widetilde T_{\bi}^{-n}\bigr ]\bigr ) \circ \kappa_{\omega}^{-1} \circ R_{n}^{-1}\ \|_{H^{t'}_p} \, .
\end{align}
We focus on the first double sum in the right hand-side (the second is similar).
For $\tau_m$ such that $\sum_m \tau_m\equiv 1$  as in the localisation  Lemma~\ref{lem:loc-prin}, set
\begin{align}
\varphi_{\omega'}^{m, n}=(\tau_{m} \circ R_{n} \circ \kappa_{\omega'})\cdot\left(\theta_{\omega'} \varphi\right)
\, , \quad m\in \integer^d\, , \, \, 
\mbox{ so that }
\theta_{\omega'} \varphi=
\sum_{m\in \integer^d} \varphi_{\omega'}^{m, n}\, . \end{align}
Since $\theta_{\omega'}$ is compactly supported, only a finite number of terms
in the above sum are nonzero.
In addition, the functions $(\tau_{m} \circ R_{n} \circ \kappa_{\omega'})_{m\in \integer^d}$ have finite intersection
multiplicity $\beta$.
Thus,  for each $\bi$ and $\omega'$, the fragmentation bound \eqref{frag} in
Lemma~\ref{lem:frag-rec} (applied to the partition of unity $\tau_m$) estimates the $p$th power of the  term 
for $(\bi,\omega')$ in  \eqref{eq:decomp-transfer2}  by (using $(|a|+|a'|)^p<2^{p-1} (|a|^p+|a'|^p)$)
\begin{align}
\label{main-task}
C_{t,p}&\beta^{(p-1)/p}\sum_{m\in \bZ^d} \|\biggl (\theta_{\omega} [( \theta_\bi{\tilde g_\bi^{(n)}}  \mathbf{1}_{O_{\bi}}\varphi_{\omega'}^{m, n} )\circ \widetilde T_{\bi}^{-n}] \biggr )\circ \kappa_{\omega}^{-1} \circ R_{n}^{-1} \|^p_{H^t_p}\\
\nonumber &+\tilde C_{t,p,\tau\circ R_m}\sum_{m\in \bZ^d} \|\biggl (\theta_{\omega} [( \theta_\bi{\tilde g_\bi^{(n)}}  \mathbf{1}_{O_{\bi}}\varphi_{\omega'}^{m, n} )\circ \widetilde T_{\bi}^{-n}] \biggr )\circ \kappa_{\omega}^{-1} \circ R_{n}^{-1} \|^p_{H^{t'}_p}\, .
\end{align}
We focus on the first term above.
The set of indices $\bi \in I^n$ for which the 
 term in \eqref{main-task} corresponding to a fixed pair $(\omega',m)$
 is non zero is contained in the set
\begin{equation}
J^n=J^n(\omega', m)=\{\bi \in I^n: {O}_{\bi} \cap \supp \varphi_{\omega'}^{m,n}\neq \emptyset\}.
\end{equation}
Since $R_n$ expands  by a factor $r_n$, while the size of the supports of
the functions $\tau_m$ is uniformly bounded, 
taking $r_n$ large enough, we can guarantee that 
$\supp\varphi_{\omega'}^{m,n}$ is small enough such that $\# J^n(\omega', m)\le D_n^b$ for each $\omega'$
and $m$.

For $\omega$, $\omega'$, and $\bi\in J^n$ such
 that $\tT_\bi^{-n}(V_\omega) \cap (V_{\omega'} \cap O_\bi)\ne \emptyset$, setting
\begin{equation*}
\begin{split}
& {\varphi}_{\omega'\bi}^{m, n}=\tau_{m}\cdot 
\bigl ( [\mathbf{1}_{O_{\bi}}  \cdot 
( \theta_{\omega'} \varphi)  ] \circ \kappa_{\omega'}^{-1} \circ R_{n}^{-1}\bigr )\\ 
&F=R_{n} \circ \kappa_{\omega'} \circ \widetilde T_{\bi}^{-n} \circ \kappa_{\omega}^{-1} \circ R_{n}^{-1}\, ,\, \,\,
f=\tilde \tau_m \cdot \bigl (\bigl (\theta_\omega [\tilde\theta_{\omega'} \theta_{\bi}\tilde g^{(n)}_\bi]\circ \widetilde T_{\bi}^{-n}\bigr )\circ \kappa_\omega^{-1} \circ R_n^{-1}\bigr ) \, ,
\end{split}
\end{equation*}
for  $\CC^\infty$ functions $\tilde\theta_{\omega'}:\widetilde M \to [0,1]$ ,
$\tilde \tau_m:\real^d \to [0,1]$ with 
$$\theta_{\omega'}(x) =\tilde\theta_{\omega'}(x)\theta_{\omega'}(x)\, , \,\,
\forall x\in \widetilde M \, , \quad 
\tilde\tau_m(u)\tau_m(u)= \tau_m(u)\, ,\,\,\forall u \in \real^d\, ,
$$
we have
\begin{equation}\label{chartsdyn}
(\theta_{\omega} [({\theta_\bi\tilde g_{\bi}^{(n)}}  \mathbf{1}_{O_{\bi}}\varphi_{\omega'}^{m, n}) \circ \widetilde T_{\bi}^{-n} ])\circ \kappa_{\omega}^{-1} \circ R_{n}^{-1}=
f \cdot ( {\varphi}_{\omega'\bi}^{m, n} \circ F)\, .
\end{equation}
Increasing $r_n$ if needed (and choosing $\tilde \tau_m$ with a small enough
support), the map $F=F(n,\bi,\omega,\omega')$ satisfies  the condition \eqref{control} of Lemma~\ref{lem:LY-local}, if we take for
$U$ the intersection of $R_n(\kappa_\omega(\tT^n_\bi(\widetilde O_{\bi}\cap V_{\omega'})))$
with the interior of the support of $\tilde \tau_m$. Then, Lemma~\ref{lem:LY-local} 
applied to the map $F$, the weight $f=f(n,m,\bi,\omega,\omega')$,
 which is $\CC^\alpha$ and supported on $U$, 
 and the test function ${\varphi}_{\omega'\bi}^{m, n}$ gives
 constants $C_n$ (depending on  $r_n$) and $C_t$ such that,
  recalling $\tilde \nu^{t}_{n,\bi}$ from \eqref{eq:hyp-indext}, and, setting
 $\Theta_\bi=\supp \theta_\bi$, 
\begin{align}
\nonumber
\|f\cdot ({\varphi}_{\omega'\bi}^{m, n}& \circ F)\|_{H_{p}^{t}}\\
\label{eq:after-locLY}&\leq C_{t} 
\sup _{\Theta_\bi }|{\tilde g_{\bi}^{(n)}}| 
\sup _{\Theta_\bi}|\det D \tT^{n}_\bi|^{\frac{1}{p}} 
\sup _{\Theta_\bi}\tilde \nu^{t}_{n,\bi}\|{\varphi}_{\omega'\bi}^{m, n}\|_{H_{p}^{t}}
+C_{n}\|{\varphi}_{\omega'\bi}^{m, n}\|_{H_{p}^{t'}}\, .
\end{align}
Next, using bounded distortion for uniformly expanding maps (see e.g. \cite[III.1]{Man}),  there exists $C>0$ such that for all $n$ and $\bi$,
\begin{equation*}
\begin{split}
\sup _{\Theta_\bi}(|\tilde {g}_{\bi}^{(n)}|) \sup _{\Theta_\bi}(|\det D \tT^{n}_\bi|^{1 / p})&
\leq C \inf _{\Theta_\bi}(|\tilde {g}_{\bi}^{(n)} \| \det D \tT^{n}_\bi|^{1 / p})\, .
\end{split}
\end{equation*}
Finally,   using $(\inf a) \cdot(\sup b) \leq \sup (a \cdot b)$,  we have
\begin{equation}\label{sup-O}
  \sup _{\Theta_\bi}|{\tilde g_{\bi}^{(n)}}| \sup _{\Theta_\bi}|\det  D \tT^{n}|^{\frac{1}{p}} \sup _{\Theta_\bi} \tilde \nu^{t}_{n,\bi}\le C \sup _{\Theta_\bi}(|{ \tilde g^{(n)}_\bi}||\det D\tT^{n}_\bi|^{\frac{1}{p}}\tilde \nu^{t}_{n, \bi})\, .
\end{equation}
Setting $\Sigma_\bi=\{ (\omega',m) \mid \bi \in J^n(\omega',m)\}$
and
\begin{equation}\label{misfact}
{\Xi}_{n,\bi}=\sup _{\Theta_\bi} [| \tilde g^{(n)}_\bi||\det D \tT^{n}_\bi|^{\frac{1}{p}} \tilde \nu^{t}_{n, \bi}]\, ,
\end{equation}
by \eqref{eq:after-locLY}, \eqref{chartsdyn}, \eqref{sup-O}, and the Minkowski inequality, we have
\begin{align}
\nonumber
\sum_{\bi \in I^n}&
\bigg(\sum_{(\omega',m)\in \Sigma_\bi}\|({\theta_{\bi}\tilde g_{\bi}^{(n)}}) \circ \widetilde T_{\bi}^{-n} \circ \kappa_{\omega}^{-1} \circ R_{n}^{-1} \cdot {\varphi}_{\omega'\bi}^{m, n} \circ F\|^p_{H_{p}^{t}}\bigg)^{\frac 1p}\\
\label{before-holder}&\le C_{t,p}\sum_{\bi}{\Xi}_{n,\bi}\bigg(\sum_{(\omega',m)\in \Sigma_\bi}\|{\varphi}_{\omega'\bi}^{m, n}\|^p_{H^t_p}\bigg)^{\frac 1p}+C_{n}\sum_{\bi}\bigg(\sum_{(\omega',m)\in \Sigma_\bi}\|{\varphi}_{\omega'\bi}^{m, n}\|^p_{H_{p}^{t'}}\bigg)^{\frac 1p }\, .
\end{align}
Let us estimate the first term above. For all $q\ge 1$, the\footnote{We use the standard notation
$(\sum_i |a_i|^{q'})^{1/q'}=\sup_i |a_i|$ if $q'=\infty$.} H\"older inequality gives
\begin{align}
\nonumber \sum_{\bi\in I^n}{\Xi}_{n,\bi}
\bigl(\sum_{(\omega',m)\in \Sigma_\bi} &\|{\varphi}_{\omega'\bi}^{m, n}\|^p_{H^t_p}\bigr)^{\frac 1p}\\
\label{Holder}&\le \bigl (\sum_{\bi}{\Xi}_{n,\bi}^{q}\bigr)^{\frac 1q}
\bigl(\sum_{\bi}\big(\sum_{(\omega',m)\in \Sigma_\bi}\|{\varphi}_{\omega'\bi}^{m, n}\|^p_{H^t_p}\big)^{\frac{q}{(q-1)p}}
\bigr)^{\frac{q-1}{q}} \, .
\end{align}
To estimate $\left\|\varphi_{\omega'\bi}^{m, n}\right\|_{H_{p}^{t}}$, we argue as in \cite{Tho} to get rid of the characteristic functions $\mathbf{1}_{O_\bi}$. On the one hand, if the support  of $\varphi_\omega^{n,m}$ is small enough, which is guaranteed if $r_n$ is large enough,
then
Lemma~\ref{lem:boundary-comp} provides a neighbourhood $O'$ of this support and a matrix $A$
such that the intersection of $R_n(A(\kappa_{\omega'} (O' \cap O_{\bi})))$ with almost any line parallel to a coordinate axis has at most $L_0 n$ connected components. Hence,
since\footnote{This is the only place where this assumption is used.} $t<1/p$, Lemma~\ref{lem:charact} applied to multiplication by $\mathbf{1}_{O'\cap O_\bi}\circ \kappa_{\omega'}^{-1}\circ A^{-1}\circ R_n^{-1}$, using that  $A$ is orthogonal and commutes with $R_n$, implies,
for all $\bi\in J^n$,
\begin{equation}
\| (\mathbf{1}_{O_{\bi}} \circ \kappa_{\omega'}^{-1} \circ R_{n}^{-1})\cdot v
 \|_{H_{p}^{t}} \leq C_{t,p} L_0 n\|
v\|_{H_{p}^{t}}\, , \forall v \, .
\end{equation}
 Thus, we obtain
\begin{equation}\label{sigma}
\|\varphi_{\omega'\bi}^{m, n}\|_{H_{p}^{t}}
 \leq C_{t,p} L_0 n\|\tau_m \cdot (\theta_{\omega'} \varphi) \circ \kappa_{\omega'}^{-1} \circ R_{n}^{-1}\|_{H_{p}^{t}}\, ,\, \forall \bi \in J^n\, .
\end{equation}
We can now estimate  the second factor  in  the right hand-side of \eqref{Holder}. Using \eqref{sigma} 
and \eqref{eq:localization} from the localisation Lemma~\ref{lem:loc-prin},
we have,
for any $q\in[ 1, \frac p {p-1}]$,
\begin{align}
\label{rhs} &\bigg(\sum_{\bi\in I^n}\bigg(\sum_{(\omega',m)\in \Sigma_\bi}\|{\varphi}_{\omega'\bi}^{m, n}\|^p_{H^t_p}\bigg)^
{\frac{q}{(q-1)p}}\bigg)^{\frac{q-1}{q}}\\
\nonumber &\,\,\,\,\le C_{t,p} n \bigg(\sum_{\bi}\big(\sum_{(\omega',m)\in \Sigma_\bi}
\|\tau_m\cdot (\theta_{\omega'} \varphi) \circ \kappa_{\omega'}^{-1} \circ R_{n}^{-1}\|^p_{H_{p}^{t}}\big)^{\frac{q}{(q-1)p}}\bigg)^{\frac{q-1}{q}}\\
\nonumber &\,\,\,\,\le C_{t,p} n \bigg(\big(\sum_{\omega',m}  D^b_n\cdot \sup_{\bi \in J^n(\omega', m)}
\|\tau_m\cdot (\theta_{\omega'} \varphi) \circ \kappa_{\omega'}^{-1} \circ R_{n}^{-1}\|^p_{H_{p}^{t}}\big)^{\frac{q}{(q-1)p}}\bigg)^{\frac{q-1}{q}}\\
\label{last-complexity}&\,\,\,\,\le   C_{t,p} n (D^b_n)^{1/p} \big(\sum_{\omega'}\| (\theta_{\omega'} \varphi) \circ \kappa_{\omega'}^{-1} \circ R_{n}^{-1}\|^p_{H_{p}^{t}}\big)^{\frac 1p}\, ,
\end{align}
exchanging the sums over $\bi$ and $(\omega',m)$ (which is legitimate
since $\frac q {(q-1)p}\ge 1$) and
using $\# J^n (\omega', m)\le D^b_n$ for
all $m$ and $\omega'$ in the penultimate line.
Combining  \eqref{last-complexity}, \eqref{Holder}, and \eqref{before-holder} with \eqref{eq:decomp-transfer2}, we obtain, for any $q\in[ 1, \frac p {p-1}]$, 
\begin{align}
\nonumber \|\left(\theta_{\omega} \mathcal{L}_{g}^{n} \varphi\right) \circ &\kappa_{\omega}^{-1} \circ R_{n}^{-1}\|_{H^t_p}
\le  C_n   \sum_{\omega'}\|(\theta_{\omega'} \varphi) \circ \kappa_{\omega'}^{-1} \circ R_{n}^{-1}\|_{H_{p}^{t'}}\\
\nonumber &+
 C_{t,p} n (D^b_n)^{1/p}\big (\sum_{\bi\in I^n}{\Xi}_{n,\bi}^{q}\big)^{\frac 1q} \big(\sum_{\omega'}\|(\theta_{\omega'} \varphi) \circ \kappa_{\omega'}^{-1} \circ R_{n}^{-1}\|^p_{H_{p}^{t}}\big)^{\frac 1p}
\, .
\end{align}
Hence, the reconstitution bound \eqref{rec} in Lemma~\ref{lem:frag-rec} applied to the
partition of unity $\theta_{\omega'}$ gives
\begin{align}
\nonumber \|\left(\theta_{\omega} \mathcal{L}_{g}^{n} \varphi\right) \circ \kappa_{\omega}^{-1}& \circ R_{n}^{-1}\|_{H^t_p}
\le C_{n}\sup_{\omega'\in \Omega}\|(\theta_{\omega'} \varphi) \circ \kappa_{\omega'}^{-1} \circ R_{n}^{-1}\|_{H_{p}^{t'}}\\
\label{changeq}&+C_{t,p}n (D^b_n)^{1/p} \big (\sum_{\bi\in I^n}{\Xi}_{n,\bi}^{q}\big)^{\frac 1q} \sum_{\omega'\in \Omega}\|(\theta_{\omega'} \varphi) \circ \kappa_{\omega'}^{-1} \circ R_{n}^{-1}\|_{H_{p}^{t}}\, .
\end{align}
Since $\Omega$ is finite and the map $q\mapsto \|v\|_{\ell^q}$
is strictly decreasing on $[1,\infty)$ for any fixed $v\in \real^D$, this implies \eqref{eq:main-LY} by
\eqref{misfact} and
the definition \eqref{eq:zoom-norm} of $\|\cdot\|_{r_n,t,p}$.
(Replacing  $\Theta_\bi$ by $O_\bi$, $\tilde g_\bi^{(n)}$ by $g^{(n)}$,
$\det D \tT_\bi^n$ by $\det DT^n$,  and $\tilde \nu_{n,\bi}$
by $\nu_n$ in \eqref{misfact} costs a factor
$(1+\epsilon)^n$, for arbitrarily small  $\epsilon>0$, up to taking $\theta_\bi$ with small
enough support.)
\end{proof}

\subsection{Proof of Theorem~\ref{thm:ess-bound} and  Corollary~\ref{lecor}}\label{sec:proof-thm}

\begin{proof}[Proof of Theorem~\ref{thm:ess-bound}]
The claim on the spectral radius on $L_p=\HH^0_p$ is an obvious
consequence on the $L_p$ bound in  Proposition~\ref{lem:main-LY}.

For the bound on the essential spectral radius, since each norm $\|\cdot\|_{r_{n_0},t,p}$ is equivalent to $\|\cdot\|_{\HH^t_p}$, and since the inclusion of $\HH^t_p$ in $L_p(M)$ is compact for $t>0$, the Lasota--Yorke bound in Proposition~\ref{lem:main-LY} implies, by a result of Hennion \cite[Cor.~1]{Hen}, that
 \begin{align}
 \label{eq:ess-spec}
r_{\text{ess}}(\LL_g|_{\HH^t_p})
&\le \rho=\E^{\frac {D^b(T)}p }\lim_{n\to \infty} \bigg(\sum_{\bi\in I^n}\sup _{O_\bi}\big(|g^{(n)}|
|\det D T^{n}|^{\frac 1p}\nu_n^t \big) ^{\frac p {p-1}}\bigg)^{\frac{p-1}{ p n}}\, .
\end{align}
(Indeed,  fix $n_0\ge 1$ very large, so that the limit in \eqref{eq:ess-spec}
is almost attained, decompose $n=\ell n_0+ n'$, with $\ell \ge 0$ and $0\le n'<n_0$,
and apply Proposition~\ref{lem:main-LY} inductively to bound  $\LL^{n}_g$ 
for the norm $\|\cdot \|_{r_{n_0},t,p}$, which is  equivalent to the norm of $\HH^t_p(M)$.)
Then  \eqref{eq:ess-spec}
implies the bound \eqref{eq:ess-bound1} for the essential spectral
radius, by definition of $P^*_{\optop}$.
\smallskip

To show \eqref{eq:ess-bound1t'}, we modify the proof 
of Proposition~\ref{lem:main-LY} as follows:
 In the definition \eqref{misfact} of $\Xi_{n,\bi}$, we
 replace $t$ by $t-s$. In \eqref{before-holder}, \eqref{Holder},
 and \eqref{rhs}, we replace $\varphi_{\omega',\bi}^{m,n}$ by
 $\sup_{\Theta_\bi}\tilde \nu_{n,\bi}^{s} \varphi_{\omega',\bi}^{m,n}$.
 In the line after \eqref{rhs}, we insert $(\sup_{\Theta_\bi}\tilde \nu_{n,\bi}^{s})^p$
 before the factor $\| \tau_m \cdot ...\|^p_{H^t_p}$.
 In \eqref{last-complexity} and the line above it, we replace $D_n^b$
 by $(1+\epsilon)^nD_n^b(\{\nu_n^{s}\})$ (with $\epsilon$
as in the end of the proof 
of Proposition~\ref{lem:main-LY}).
 \end{proof}

\begin{proof}[Proof of Corollary~\ref{lecor}]
Since \eqref{changeq} holds for any $q\in [1, \frac p {p-1}]$, the
conclusion of  Proposition~\ref{lem:main-LY}
holds replacing $\bigl( \sum_{\bi\in  I^n}  \sup _{O_\bi} f^{\frac p {p-1}}_{n,t,p}\bigr)^{\frac{p-1}{p}}$
by the infimum  of
$$\bigl( \sum_{\bi\in  I^n}  \sup _{O_\bi} f^{q}_{n,t,p}\bigr)^{\frac 1 q}$$
over such $q$. Thus \eqref{eq:ess-bound1}
in Theorem~\ref{thm:ess-bound}
holds replacing $\frac {p-1} {p}  P^*_{\optop}(\{\frac p {p-1} \log f_{n,t,p}\})$ by
the infimum over such $q$ of  $\frac {1} {q}  P^*_{\optop}(\{q\log f_{n,t,p}\})$.
The bound \eqref{eq:ess-bound2} follows from this version of \eqref{eq:ess-bound1}
 combined with  Theorem~\ref{thm:var-princ} with 
$G
=g|\det DT|^{\frac 1p}$:  We have, for  $q\in[ 1, \frac p {p-1}]$,
\[
\begin{split}
\log r_{\text{ess}}(\LL_g|_{\HH^t_p})&-\frac {D^b(T)}p\le \frac{1}{q}
 P^*_{\optop} (q\log (|G^{(n)}| \nu_n^{t}))\\
&=\frac 1q\sup_{\mu\in \Erg(T)}
\left\{ h_{\mu}(T)+\int q\log |G|\D \mu-tq\chi_{\mu}(D T)\right\}\\
&=\sup_{\mu\in \Erg(T)}\left\{ \frac 1qh_\mu(T)+\int \log |g(\det DT)^{\frac 1p}| \D \mu-t\chi_\mu(DT)\right\}\, .
\end{split}
\]
The final claim is shown just like the second bound of Corollary~\ref{cor1}.
\end{proof}

\section{Proof of Theorem~\ref{thm:var-princ} on the Subadditive Variational Principle}\label{sec:pressure}
We show  Theorem~\ref{thm:var-princ} in \S\ref{both}, adapting the proof\footnote{Applying directly \cite[Thm~1.3]{Buz}
and using Lemma~\ref{copyBT} would also give Theorem~\ref{thm:var-princ}, along the lines 
of \cite[Lemma~B.6]{Babook2}. (Note that
small boundary pressure of $T^m$ for $\log f_{m,t}$ for all large enough $m$
is equivalent to the condition in Theorem~\ref{thm:var-princ}
by Lemma~\ref{copyBT}.)}  of \cite[Thm~3.1(i)]{Buz} to subadditive potentials. 
For this, we first state and prove a key proposition about measures with $\mu(\SS_\OO)>0$
in \S\ref{>0} 
and next recall in \S\ref{symbd} the symbolic dynamics of a piecewise expanding
map and a variational principle of Cao--Feng--Huang.

\subsection{Measures Giving Nonzero Mass to $\SS_\OO$}\label{>0}
The proof of  Theorem~\ref{thm:var-princ} is based on the following proposition,  inspired  from \cite[Prop.~3.1]{Buz}:

\begin{prop}\label{prop:key}
Let $T$ be a piecewise $\CC^{\bar\alpha}$ expanding map. Recall
$\SS_\OO$  
from \eqref{singset}.
For each $\mu\in \Erg(T)$ such that $\mu(\SS_\OO)>0$, every $t\ge 0$,
and each piecewise $\CC^\alpha$ function $G:M\to \real^+_*$, we have
\[
h_{\mu}(T)+\int  \log G \D \mu-t\chi_{\mu}(D T)\le   P^*_{\optop}(\{\log(G^{(n)}\nu_n^t)\},\SS_\OO)\, .
\]
\end{prop}

\noindent Our proof  of the proposition uses the following  lemma
  (see \cite{BaTsu1}, \cite[Lemma~B.3]{Babook2}):
 \begin{lemma}\label{copyBT}
 Let $T$ be a piecewise $\CC^{\bar\alpha}$ expanding map, and let $G:M\to \real$ be piecewise
 continuous. Then, for any measurable set $E\subset M$ and any $t\ge 0$,
 $$
 P^*_{\optop}(T,\{ \log (|G^{(n)}| \nu_n^t) \mid n\ge 1\}, E)=
 \lim_{m \to \infty} \frac 1 m P^*_{\optop}(T^m, \log (|G^{(m)}| \nu_m^t)  , E)\, .
 $$
 \end{lemma}
 
 \begin{proof}
  The limit in the right-hand side exists in $\real \cup \{-\infty\}$
 by submultiplicativity. By definition, for each $\epsilon >0$, there exists
 $m\ge 1$ such that
 \begin{align*}
 P^*_{\optop}(T,\{ \log (|G^{(n)}| \nu_n^t) \mid n\ge 1\}, E)+\epsilon 
 &\ge\frac 1 m \log  \sum_{\bi \in I^m\mid E\cap O_\bi \ne \emptyset} \sup_{O_\bi} (|G^{(m)}| \nu_m^t)
 \\
 &\ge \frac 1 m P^*_{\optop}(T^m, \log (|G^{(m)}| \nu_m^t) , E)\, .
 \end{align*}
 (we used that the limit defining $P^*_{\optop}(T^m, \log f , E)$
 is an infimum).
 Thus
 $$P^*_{\optop}(T,\{ \log (|G^{(n)}| \nu_n^t) \mid n\ge 1\}, E)\ge
 \lim_{m \to \infty} \frac 1 m P^*_{\optop}(T^m, \log (|G^{(m)}| \nu_m^t)  , E)\, .
 $$
 The other inequality follows from submultiplicativity, which gives, for any  $m>0$, 
\begin{align*}
&P^*_{\optop}(T,\{ \log (|G^{(n)}| \nu_n^t) \mid n\ge 1\}, E)= 
\lim_{k\to \infty}
\frac{1}{mk}
 \log  \sum_{\bi \in I^{mk}\mid E\cap O_\bi \ne \emptyset} \sup_{O_\bi} (|G^{(mk)}| \nu_{mk}^t)
 \\
&\qquad\qquad\qquad\le 
\frac{1}{m} \lim_{k\to \infty} \frac 1 k
\log  \sum_{\bi \in (I^{m})^k\mid E\cap O_\bi \ne \emptyset} \sup_{O_\bi} (|G^{(m)}|^{(k)} \nu_{m}^{tk})
 \\
&\qquad\qquad\qquad=
\frac 1 m P^*_{\optop}(T^m, \log (|G^{(m)}| \nu_m^t)  , E)
 \, .
\end{align*}
 \end{proof}
 
\begin{proof}[Proof of Proposition~\ref{prop:key}]
We  may assume $\inf G>0$ because, if $G_k$ is a  sequence of  piecewise continuous
functions  such that
$\inf_M G_k>0$, with  $G_{k}\ge G_{k+1}\ge |G|$ for all
$k$, and $\lim_{k\to\infty}\|G_k- |G|\|_{L_\infty(M)}=0$,
 then for  any measurable set $E$, 
applying our definitions to the sequences
$f_n=\log (G^{(n)} \nu_n^t)$ and (for a fixed $k$) $f'_n=\log (G_k^{(n)} \nu_n^t)$,
\begin{equation}\label{zeroclaim}
\lim_{k\to \infty}P^*_{\optop}(T,\{ \log (G_k^{(n)} \nu_n^t) \}, E)=
P^*_{\optop}(T,\{ \log (|G^{(n)}| \nu_n^t) \}, E)\, , \,\forall t\ge 0\, .
\end{equation}
To show \eqref{zeroclaim}, it suffices to show  that for all $t\ge 0$
(see\footnote{There are typos in the proof there: $G_n$ should be replaced by $G$ (twice) in the third line
of that proof, and $Q_*(T,G, \lambda^{(*)},\WW)+2\epsilon$ should be
$Q_*(T,G, \lambda^{(*)},\WW,m)+\epsilon$
in the 5th line.} \cite[Lemma~B.4]{Babook2})
$$
\lim_{k \to \infty} P^*_{\optop}(T,\{ \log (G_k^{(n)} \nu_n^t) \mid n\ge 1\}, E)\le
P^*_{\optop}(T,\{ \log (|G^{(n)}| \nu_n^t) \mid n\ge 1\}, E) \, .
$$ 
Fix $E$ and $t$. For any $\epsilon>0$, there exists $m=m(\epsilon)$ large enough such that 
$$
P^*_{\optop}(T,\{ \log (|G^{(n)} |\nu_n^t) \mid n\ge 1\}, E)+\epsilon
\ge \frac 1 m
\log  \sum_{\bi \in I^m\mid E\cap O_\bi \ne \emptyset} \sup_{O_\bi} (|G^{(m)}| \nu_m^t)\, .
$$
Then take $k_0(m)$ such that for all $k\ge k_0$
$$
\frac 1 m
\log  \sum_{\bi \in I^m\mid E\cap O_\bi \ne \emptyset} \sup_{O_\bi} (G_k^{(m)} \nu_m^t)
\le \frac 1 m\bigl ( \epsilon+
\log  \sum_{\bi \in I^m\mid E\cap O_\bi \ne \emptyset} \sup_{O_\bi} (|G^{(m)}| \nu_m^t))
 \, .
$$ 
We conclude the proof of \eqref{zeroclaim} by  submultiplicativity: For all $k\ge k_0$,
$$P^*_{\optop}(T,\{ \log (G_k^{(n)} \nu_n^t) \mid n\ge 1\}, E)\le
P^*_{\optop}(T,\{ \log (|G^{(n)}| \nu_n^t) \mid n\ge 1\}, E)+2\epsilon  \, .
$$

So let us assume that $\inf |G|>0$. By definition,
$$
(G^{(n)} \nu_n^t)(x)=\exp\left(\sum_{k=0}^{n-1}\log G(T^k(x))
-t \log \inf_{ \|v\|=1} \|D_{ x} T ^n (v) \|\right)\, .
$$
We introduce a generalisation of \eqref{defnn}, setting, for $n, m\ge 1$,
$$
f^{(n,m)}(x)=\prod_{k=0}^{n-1} f(T^{k m}(x)) \, , \,\, x\in M\, .
$$
Fix $\epsilon>0$.
First, by Lemma~\ref{copyBT} there exist $m_0\ge 1$ and  a sequence $n_0(m)\ge 1$  such that 
(using the convention that the supremum of any function over the empty set is zero)
\begin{align}
\nonumber \sum_{\bi\in I^{n m}}\sup_{O_{\bi}}(G^{(mn)}(\nu_{m}^t)^{(n,m)})&
\le \E^{{n \left(P^*_{\optop}(T^m, \log (G^{(m)}\nu_{m}^t),\SS_\OO)+\epsilon\right)}}\, , \, \forall m \ge 1\, ,
\forall n\ge n_0(m) ,\\
\label{one-side}&
\le \E^{{nm \left(P^*_{\optop}(\{\log (G^{(k)}\nu_{k}^t\},\SS_\OO)+2\epsilon\right)}}\, , \, \forall m \ge m_0\, ,
n\ge n_0(m)  .
\end{align}

Next, for any  $\mu\in \Erg(T)$, Oseledec's theorem \cite{Via} implies
\begin{equation}
\lim_{m\to \infty}  \frac{t}{m} \log \nu_m(x) =-t \chi_{\mu}(DT)\, ,
\quad  \mbox{ for } \mu \mbox{ almost every } x\, .
\end{equation}
Thus,  by the Birkhoff ergodic theorem, there exists a
set $R\subset M$, with $\mu(R)>1- \frac {\mu(\SS_\OO)}2$,
and there exists an integer $m_0\le m_1(\epsilon)<\infty$  such that 
\begin{equation}
(G^{(m_1)} \nu_{m_1}^t)(x)\ge \E^{m_1\left(\int \log G \D\mu-t\chi_\mu(DT)-\epsilon\right)}\, ,\,
\forall x \in R
\, .
\end{equation}
Therefore
\begin{equation}\label{oseledec}
(G^{(m_1)} \nu_{m_1}^t)^{(n,m_1)} (x)\ge \E^{n m_1\left(\int \log G \D\mu-t\chi_\mu(DT)-\epsilon\right)}\, ,\,
\forall x \in R\, ,\, \forall n\ge 1
\, .
\end{equation}
For each $n\ge 1$, define 
\[
 K_n=\{\bi\in I^{nm_1}\mid O_{\bi}\cap R \cap \SS_\OO\ne \emptyset\}  \, .
\]
Since $\KK_n=\cup_{\bi \in  K_n} O_\bi$ contains $R\cap \SS_\OO$, we have $
\inf_n \mu(\KK_n)\ge \mu(R\cap \SS_\OO)>\frac {\mu(\SS_\OO)}2>0$.
Next, since $\log G$ is\footnote{We do not see why piecewise uniform
continuity of  $\phi$  suffices for \cite[(5)]{Buz}.} piecewise H\"older,
and   $\diam(\OO^{(n)})\le \diam (M)/\lambda ^n$,
there exist $C_G<\infty$ and  $n_1(\epsilon)\ge 1$ such that, for all $n\ge  n_1$
and all $\bi \in I^{nm_1}$,
\begin{equation}\label{unif-cont}
|\log G^{(nm_1)}(x)- \log G^{(nm_1)}(y)|\le C_G (\diam (\OO))^\alpha \le n \frac \epsilon 2\,  ,\, 
\, \forall x,y\in O_{\bi} \, ,
\end{equation}
and, in addition 
since $\nu_{m_1}$ is  H\"older on  $O_\bj$
for each $\bj\in I^{m_1}$,  we have for all $n\ge  n_1$,
\begin{equation}\label{unif-cont'}
t |\log (\nu_{m_1})^{(n, m_1)}(x)- \log (\nu_{m_1})^{(n,m_1)}(y)| \le n \frac \epsilon 2\,  ,\,
\, \forall x,y\in O_{\bi}\, ,\,\, \forall   \bi \in I^{nm_1}\, .
\end{equation}
It follows from (\ref{unif-cont}--\ref{unif-cont'}) and \eqref{oseledec} that, for all $n\ge n_1$,
\begin{align}
\sum_{\bi\in K_n}\sup_{x \in O_{\bi}} (G^{(nm_1)} (\nu_{m_1}^t)^{(n)})(x)&
\ge \sum_{\bi\in K_n} \E^{-n\epsilon} \sup_{x \in O_{\bi}\cap R}(G^{(m_1)}\nu_{m_1}^t)^{(n)}(x) \\
\nonumber &\ge \# K_n \cdot  \E^{nm_1\left(\int \log G \D\mu -t\chi_\mu(DT)-2\epsilon\right)}\, .
\end{align}

 Therefore, since $m_1\ge m_0$, recalling \eqref{one-side}
we have, for all $n\ge \max\{ n_0(m_0), n_1\}$,
\begin{equation}\label{cardin}
\# K_n \le C\E^{nm_1\left(P^*_{\optop}(\{\log(G^{(n)}\nu_n^t)\}, \SS_\OO)+2\epsilon\right)}
 \cdot  \E^{-nm_1\left(\int \log G \D\mu-t\chi_\mu(DT)-2\epsilon\right)}\, .
\end{equation}
Rudolph's formula for the entropy (see \cite[\S 5.1,  \S 5.10]{Rud}) says that  if $\mu\in \Erg(T)$  then, 
for any fixed $\gamma \in (0,1)$ and any finite generator, denoting by $K'_\ell$ the minimal cardinality of a collection of $\ell$-cylinders  whose union
has  measure at least $\gamma$,  we have
$ h_{\mu}(T)=\liminf _{\ell \rightarrow \infty} \frac{1}{\ell} \log K'_{\ell}$.
Therefore,  taking $\OO$ as a generator and $\gamma=\frac {\mu(\SS_\OO)} 2$, we have $\# K_n\ge K'_{nm_1}$ so that
\begin{align*}
h_{\mu}(T) &\le \liminf _{n \rightarrow \infty} \frac{1}{nm_1} \log \# K_{n}\\
&\le P^*_{\optop}(\{\log(G^{(n)}\nu_n^t)\}, \SS_\OO)+2\epsilon- \int \log G d\mu +t\chi_\mu(DT)+2\epsilon\, ,
\end{align*}
where we used \eqref{cardin} for the second inequality.  To conclude, let  $\epsilon \to 0$.
\end{proof}

\subsection{Symbolic Dynamics. Continuous Subadditive Variational Principle}\label{symbd}
We use the   symbolic dynamics  for a piecewise
$\CC^{\bar\alpha}$ expanding map $T$ from \cite[Beginning of \S3]{Buz}: Set
\[
\Sigma_0(T)=\{\bi_{\infty}=(i_0,i_1,\dots)\in I^{\integer_+}\mid  \exists x\in M \mid  T^n x\in O_{i_n} \, , \forall n\ge 0 \}\, ,
\]
and let   $\Sigma(T)$ be the closure of $\Sigma_0(T)$ in $I^{\integer_+}$
for  the product topology of the discrete topology on $I$. (This topology is compatible with the distance  
$\operatorname{dist}(\bi_{\infty}, \bj_{\infty})=2^{-n}$, where $n(\bi_{\infty}, \bj_{\infty})=\min\{k\in \integer_+\mid \bi^{k}_\infty\neq \bj^k_{\infty}\}$, 
where $\bi^k_{\infty}=(i_0,i_1,\dots,i_{k-1})\in I^k$,
and with the convention $\min \emptyset =\infty$.)  Let $\sigma:\Sigma(T)\to \Sigma(T)$ be the left-shift on $\Sigma(T)$.

By compactness of $M$, and since $T$ is piecewise expanding, for each $\bi_\infty\in \Sigma(T)$ there exists a unique $x\in M$ such that $\cap_{n\ge 1}\overline{O}_{\bi^n_\infty}=\{x\}$. This defines a map $\pi:\Sigma(T)\to M$ by $\pi(\bi_\infty)=\cap_{n\ge 1}\overline{O}_{\bi_\infty^n}$. Setting $\SS_\OO^\pi=\pi^{-1}(\SS_\OO)$ and $\partial \OO^\pi=\pi^{-1}(\partial \OO)$, it is easy to check that  the restriction
\[
\pi_*: \Sigma(T) \setminus \bigcup_{k \geq 0} \sigma^{-k}(\partial \OO^\pi) =\Sigma(T) \setminus \SS_\OO^\pi
\rightarrow M \setminus \bigcup_{k \geq 0} T^{-k} (\partial \OO)=M\setminus \SS_\OO
\]
is measurable, and bijective, with measurable inverse, and we have $\pi_*\circ \sigma=T\circ \pi_*$. 
Note that $\sigma$ is a continuous transformation of the compact metric space $\Sigma(T)$.

Next, given a function   $f:M\to \real$, we define  $f^\pi:\Sigma(T)\to \real$ by
\begin{equation}\label{rel-pot0}
f^\pi(\bi_\infty)=\lim_{m\to \infty}\inf_{O_{\bi^m_\infty}} f \, .
\end{equation} 
Then we have
\begin{equation}
f^\pi (\bi_\infty)=f(\pi(\bi_\infty)) \, ,\quad \forall \bi_\infty :  \pi(\bi_\infty)\notin \partial \OO \, .
\end{equation}
Moreover, for each $\alpha>0$ there exists $\alpha'$ such that
if $f$ is piecewise $\alpha$-H\"older, respectively continuous, on $M$
  then $f^\pi$ is $\alpha'$-H\"older,  respectively continuous, on $\Sigma(T)$.
  If $\phi :\Sigma(T)\to \real_+^*$ is continuous then 
  the topological pressure $P_{\optop}(\sigma, \log \phi)$ 
  is well-defined. If $f$ is such that $f^\pi:\Sigma(T)\to \real_+^*$ is continuous, then
\[
P_{\optop}(\sigma, \log f^\pi)=P^*_{\optop}(T, \log f)\, .
\]
 More generally, if $f_n:M\to \real_+^*$ is a  submultiplicative sequence of 
functions with each $f^\pi_n$ continuous, it is easy to see that
\begin{equation}\label{easy}
P_{\optop}(\sigma, \{\log f_n^\pi\})=P^*_{\optop}(T,\{ \log f_n\}) \, ,
\end{equation}
where the topological pressure  in the left-hand side is defined using $(n,\epsilon)$-separated
sets for continuous transformations of compact metric spaces in \cite[p.~640]{CFH}, or
in the case of left-shift $\sigma$ 
(see \cite[\S4, p.~649]{CFH})
using cylinders.

The topological entropy of $\sigma$  is finite. Thus, for  
a submultiplicative sequence  of continuous functions $\phi_n:\Sigma(T)\to\real_+$,  the variational principle
 in \cite[Thm~1.1 and \S4]{CFH} says that (the limit below exists by subadditivity)
\begin{equation}\label{var-prin-shift}
P_{\optop}(\sigma, \{\log \phi_n\})=\sup_{\mu_\sigma \in \Erg(\sigma)} \{ h_{\mu_\sigma}(\sigma)+
\lim_{n\to \infty} \frac 1n\int \log \phi_n \D\mu_\sigma\}\, .
\end{equation}

\subsection{Proof of Theorem~\ref{thm:var-princ}}\label{both}
\begin{proof}[Proof of Theorem~\ref{thm:var-princ}]
Fix $t\ge 0$. 
In  Step I, we shall prove that\footnote{This does not require small boundary pressure.}
$$\sup_{\mu\in \Erg(T)} h_{\mu}(T)+\int  \log |G| \D \mu-t\chi_{\mu}(D T)
\le\log  P^*_{\optop}(\{\log (G^{(n)}\nu_n^t)\})\, .$$
In  Step II, we shall find    $\mu_{0,t}\in \Erg(T)$ with
$$
h_{\mu_{0,t}}(T)+\int  \log |G| \D \mu_{0,t}-t\chi_{\mu_{0,t}}(D T)=\log  P^*_{\optop}(\{\log (G^{(n)}\nu_n^t)\})\, .
$$ 
Both steps will use Proposition~\ref{prop:key}.

We can assume $\inf |G|>0$ by \eqref{zeroclaim}.
We associate  the sequence of continuous functions $\{\log f_{n,t}^\pi\}$ 
to  $\{\log f_{n,t}=\log (|G^{(n)}|\nu_n^t)\}$ via \eqref{rel-pot0}.

We start with Step I.  Let $\mu \in \Erg(T)$.
Assume first that $\mu(\SS_\OO)=0$. Then,
\[\pi:(\Sigma(T), 
\mu \circ \pi, \sigma) 
\rightarrow (M, 
\mu, T) \, , \] 
(for the Borel $\sigma$ algebras of $\Sigma(T)$, $M$) is a  measure-theoretic isomorphism. Thus
 \[
 h_{\mu}(T)+\int \log |G| d\mu-t\chi_\mu(DT)=
 h_{\mu\circ \pi^{-1}}(\sigma)+\lim_n \frac 1 n\int \log f_{n,t}^\pi \D\mu\circ \pi^{-1}
 \, .\]
Next, by 
the variational principle \eqref{var-prin-shift},
\[
h_{\mu\circ \pi^{-1}}(\sigma)+ \lim_n \frac 1 n
\int \log  f_{n,t}^\pi \D\mu\circ \pi^{-1}\le P_{\optop}(\sigma,\{ \log f_{n,t}^\pi\})\, .
\]
Therefore, since $P_{\optop}(\sigma, \{\log f_{n,t}^\pi\})=P^*_{\optop}(T,\{\log f_{n,t}\})$ by \eqref{easy}, we have 
$$
 h_{\mu}(T)+
 \int \log |G| \D\mu - t\chi_\mu(DT)\le P^*_{\optop}( T, \{\log f_{n,t}\}) \, .
$$
Finally, if $\mu(\SS_\OO)>0$, then Proposition~\ref{prop:key}  gives
\begin{align*}
h_{\mu}(T)+
\int \log |G| \D\mu- t\chi_\mu(DT)&\le P^*_{\optop}(T,\{\log f_{n,t}\},\SS_\OO)
\le P^*_{\optop}(T,\{\log f_{n,t}\})\, .
\end{align*} 

We move to  Step II.
Since $\sigma$ is expansive and $\Sigma(T)$ is compact, the functions 
$$\mu_\sigma\mapsto h_{\mu_\sigma}(\sigma^n)
\mbox{ and } \mu_\sigma\mapsto  \lim_n \frac 1 n \int \log f_{n,t}^\pi\  \D\mu_\sigma
\, , \qquad \mu_\sigma \in \Erg(\sigma)
$$ 
are upper semi-continuous  (indeed $-\int \lim_n 
\frac 1 n  \log \nu_{n}^\pi \ \D\mu>0$ is lower semi-con\-tinuous, as it is
the smallest Lyapunov exponent $\chi_\mu(DT)$, see e.g. \cite[Lemma~9.1]{Via}, hence $ \lim_n 
\frac 1 n  \int \log \nu_{n}^\pi \ \D \mu<0$ and $\lim_n 
\frac 1 n  \int \log f_{n,t}^\pi\ \D \mu$ are upper semi-continuous). Therefore, there exists $\mu_{0}\in \Erg(\sigma)$ with
\begin{align*}
h_{\mu_{0}}(\sigma)+\lim_n \frac 1 n\int \log f_{n,t}^\pi \ \D\mu_{0}&=\sup_{\mu_\sigma\in \Erg(\sigma)}h_{\mu_\sigma}(\sigma)
+\lim_n \frac 1 n\int \log f_{n,t}^\pi\ \D\mu_\sigma\\
&=P_{\optop}(\sigma,\{ \log f_{n,t}^\pi\} )
\, ,
\end{align*}
where the second inequality follows from the  variational principle \eqref{var-prin-shift}. 

Setting $\mu_{0,t}=\mu_0\circ \pi^{-1}$, suppose that $\mu_{0,t}(\SS_\OO)>0$, so that $\mu_0(\pi^{-1}(\SS_\OO))>0$. Then  Proposition~\ref{prop:key} and the small boundary condition \eqref{eq:small-boundary} would imply
\begin{align*}
P_{\optop}(\sigma, \{\log f_{n,t}^\pi\})&=h_{\mu_0}(\sigma)+\lim_n \frac 1 n\int \log f_{n,t}^\pi \D\mu_0
\le P_{\optop}(\sigma,\{\log f_{n,t}^\pi\}, \pi^{-1}\partial \OO)\\
&\le P^*_{\optop}(T,\{\log f_{n,t}\}, \partial \OO)<
P^*_{\optop}(T,\{\log f_{n,t}\})\, .
\end{align*}
This would contradict $P^*_{\optop}(T,\{\log f_{n,t}\})=P_{\optop}(\sigma, \{\log f_{n,t}^\pi\})$.
Thus,  $\mu_{0,t}(\SS_\OO)=0$ and, arguing as in Step I, we have  
\begin{align*}
h_{\mu_{0,t}}(T) +\lim_n \frac 1 n\int \log f_{n,t}\D\mu_{0,t}
&=h_{\mu_0}(\sigma)+\lim_n \frac 1 n\int \log f_{n,t}^\pi \D\mu_0  \, .
\end{align*}
The right-hand side is  $P_{\optop}(\sigma,\{ \log f_{n,t}^\pi\})=P^*_{\optop}(T,\{ \log f_{n,t}\})$,  and
 we conclude. 
 \end{proof}


\end{document}